\documentclass[11pt,a4paper,reqno]{amsart}
\usepackage[applemac]{inputenc}
\usepackage[T1]{fontenc}
\usepackage{amsmath}
\usepackage{amsthm}
\usepackage{amsfonts}
\usepackage{amssymb}
\usepackage{graphicx}
\usepackage{color}
\usepackage{xcolor}
\usepackage{amsbsy}
\usepackage{mathrsfs}
\usepackage{bbm}
\usepackage{hyperref}
\usepackage{mathtools}
\usepackage{caption}
\usepackage{subcaption}
\newtheorem{theorem}{Theorem}[section]
\newtheorem{lemma}[theorem]{Lemma}

\newtheorem{proposition}[theorem]{Proposition}
\newtheorem{corollary}[theorem]{Corollary}
\newtheorem{definition}[theorem]{Definition}
\theoremstyle{remark}

\numberwithin{equation}{section}
\catcode`@=11
\def\section{\@startsection{section}{1}%
  \z@{1.5\linespacing\@plus\linespacing}{.5\linespacing}%
  {\normalfont\bfseries\large\centering}}
\catcode`@=12
\newcommand{\be}{\begin{equation}}
	\newcommand{\ee}{\end{equation}}
\newcommand{\bea}{\begin{eqnarray}}
	\newcommand{\eea}{\end{eqnarray}}
\newcommand{\bee}{\begin{eqnarray*}}
	\newcommand{\eee}{\end{eqnarray*}}

\def\pa{\partial}

\def\CC{\mathbb{C}}
\def\NN{\mathbb{N}}
\def\RR{\mathbb{R}}

\def\SS{\mathcal{S}}

\def\ut{\tilde{u}}
\def\wt{\tilde{w}}

\def\calA{{\mathcal A}}

\def\calH{{\mathcal H}}

\def\calO{{\mathcal O}}

\def\calS{{\mathcal S}}

\def\calV{{\mathcal V}}
\def\calW{{\mathcal W}}

\def\calV{{\mathcal V}}

\def\calS{{\mathcal S}}

\catcode`@=11
\def\supess{\mathop{\operator@font Sup\,ess}}
\catcode`@=12

\def\Qtb{\tilde{Q}_b}

\def\CC{\mathbb{C}}

\def\NN{\mathbb{N}}

\def\RR{\mathbb{R}}
\def\SS{\mathbb{S}}

\def\a{\alpha}
\def\e{\varepsilon}

\def\bar#1{{\overline #1}}

\def\R2+{\RR ^2_+}

\def\pa{\partial}

\def\lim{\mathop{\rm lim}}

\def\supp{{\rm supp}~}
\def\sup{\mathop{\rm sup}}

\def\l{\lambda}

\def\log{{\rm log}}

\def\ut{\tilde{u}}

\def\pa{\partial}

\def\pa{\partial}
\def\la{\langle}

\def\ra{\rangle}

\def\S{\mathcal S}

\begin{document}

\title[]{Mode stability for self-similar blowup of slightly supercritical NLS: II. high-energy spectrum}


\author[Z. Li]{Zexing Li}
\address{Laboratoire Analyse, G\'eom\'etrie et Mod\'elisation,
CY Cergy Paris Universit\'e, 2 avenue Adolphe Chauvin, 95300, Pontoise, France}
\email{zexing.li@cyu.fr}

\maketitle

\begin{abstract} 
    In continuation of the study in \cite{Limodestability1}, we prove the high-energy mode stability for linearized operator around self-similar profiles in \cite{MR4250747} for slightly mass-supercritical NLS in $1 \le d \le 10$. This concludes the asymptotic stability of such self-similar blowup, and answers the question from \cite{MR4250747} and \cite{MR2729284}. As a byproduct, we characterize the spectrum for linearized operator around mass-critical ground state for $1 \le d \le 10$, which could be useful for future studies of asymptotic behavior near ground state.  

    The core idea is a linear Liouville argument, originated by Martel-Merle \cite{MR1826966,MR1753061} studying soliton stability, to reformulate the eigen problem as rigidity of linear dynamics so as to introduce modulation and to apply nonlinear dynamical controls. Our controlling quantities were verified with numerical help in \cite{MR2150386,MR2252148,MR3841347} for $1 \le d \le 10$.
\end{abstract}


\section{Introduction}

\subsection{Setting and background}
We consider as in \cite{Limodestability1} the nonlinear Schr\"odinger equation
\be \left\{\begin{array}{l} i\pa_t u+\Delta u+u|u|^{p-1}=0\\ u_{|t=0}=u_0\end{array}\right. \tag{NLS} \label{eqNLS}\ee 
for spatial dimension $d \ge 1$ and $p > 1$.
In this article, 
we are interested in the asymptotic stability of self-similar singularity for nonlinearities slightly above the mass conservation law:
\be
\label{ranges}
0<s_c := \frac d2 - \frac{2}{p-1} \ll 1.
\ee
The existence and stability of such dynamics in this range with $1 \le d \le 10$ were firstly verified by Merle-Rapha\"el-Szeftel \cite{MR2729284}.

\mbox{}

Let us recall the setting and main results of \cite{Limodestability1}. 

Let $Q_b$ with $b = b(s_c)$ be the finite-energy self-similar profile constructed in Bahri-Martel-Rapha\"el \cite{MR4250747} (its existence and asymptotics are recorded in Proposition \ref{propQbasymp}), solving the elliptic equation 
\be \Delta Q_b-Q_b+ib\left(\frac{2}{p-1}Q_b+y \cdot \nabla Q_b\right)+Q_b|Q_b|^{p-1}=0.\label{eqselfsimilar} \ee
Then it generates a \textit{self-similar} blowup solution of \eqref{eqNLS}
\[ u(t, x) = (1-2bt)^{-\frac{1}{p-1}}  Q_b\left( \frac{x}{\sqrt{1-2bt}} \right)e^{-i\frac{1}{2b}\ln\left(\frac{1}{2b} - t\right)}. \]
To study its asymptotic stability, we renormalize \eqref{eqNLS} in the self-similar coordinates where $Q_b$ is a stationary solution, and consider its linearized operator
\be \calH_b = \left( \begin{array}{cc} \Delta_b -1 &  \\ &  -\Delta_{-b} +1\end{array} \right) -ibs_c + \left( \begin{array}{cc} W_{1, b} & W_{2, b}  \\ -\overline{W_{2, b}} & -W_{1, b}   \end{array} \right),\label{eqdefcalH} \ee
	with 
\be  \Delta_b = \Delta + ib \Lambda_0, \quad \Lambda_0 := \frac d2 + x\cdot\nabla \label{eqdefDeltabLambda0} \ee
\be  W_{1, b} = \frac{p+1}{2}|Q_b|^{p-1}, \quad W_{2, b} = \frac{p-1}{2} |Q_b|^{p-3}Q_b^2. \label{eqdefW1bW2b} \ee 

Notice that as $b \to 0$, $\calH_b$ degenerates to  $\calH_0$
\bea
 \calH_0 = \left(\begin{array}{cc}
     \Delta - 1 &  \\
      & -\Delta + 1
 \end{array}  \right) + \left(\begin{array}{cc}
     W_1 & W_2 \\
     -W_2 & -W_1
 \end{array}  \right) \label{eqdefH0} \\
 W_1 = \frac{p_0+1}{2} Q^{p_0-1},\quad W_2 = \frac{p_0-1}{2} Q^{p_0-1}. \label{eqdefW1W20}
\eea
where $p_0 = \frac 4d + 1$ and $Q$ is the mass-critical ground state, namely the positive radial $H^1$ solution of 
\be \Delta Q - Q + Q^{p_0} = 0\quad {\rm on\,\,} \RR^d. \label{eqgroundstatemasscritical} \ee 
In particular, $\calH_0$ is the linearized operator around $Q$ in \eqref{eqNLS}. 
The generalized kernel of $\calH_0$ was characterized in classical works \cite{MR0783974} and \cite{kwong1989uniqueness} (also see \cite[Section 2.1]{chang2008spectra}), which verifies the Figure \ref{fig:coffee}(A) near the origin. That is the starting point of our bifurcation analysis for $\calH_b$. 

The main result of \cite{Limodestability1} characterizes the spectrum of $\calH_b\big|_{(\dot H^{\sigma})^2}$ in a neighborhood of origin with $0 < \sigma - s_c \ll 1$ as indicated in Figure \ref{fig:coffee}(B). In particular, the only unstable eigenvalues below the essential spectrum $\sigma_{\rm ess}(\calH_b)\big|_{(\dot H^{\sigma})^2} = \RR + ib(\sigma -s_c)$ with $|\l| \ll 1$ are $0, -bi, -2bi$, with eigenfunctions generated by phase rotation, spatial translation and scaling symmetries.

\begin{figure}[h!]
		\centering
		\begin{subfigure}[b]{0.45\linewidth}
			\includegraphics[width=\linewidth]{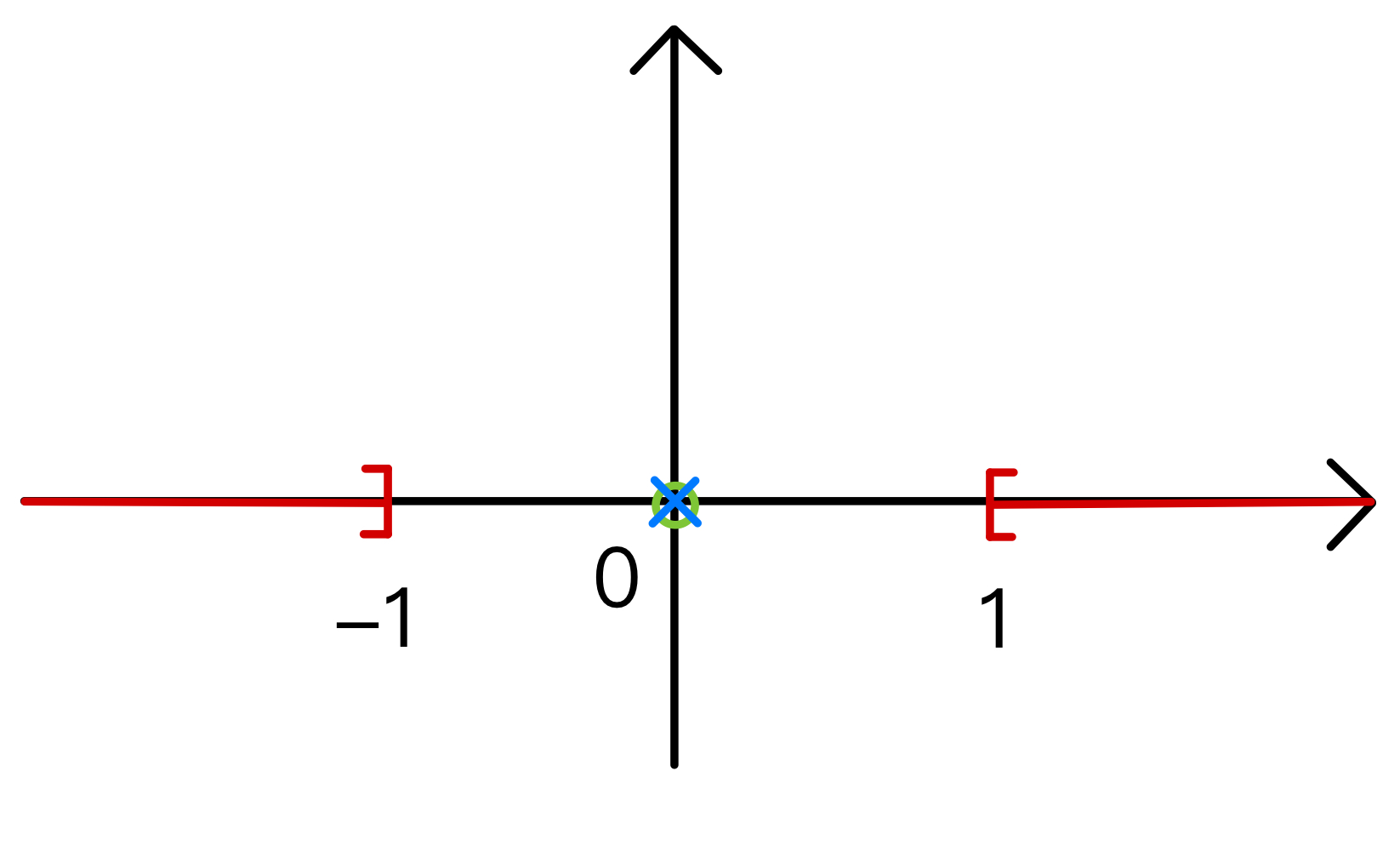}
			\caption{$\sigma(\calH_0 \big|_{(L^2)^2})$}
		\end{subfigure}
		\begin{subfigure}[b]{0.45\linewidth}
			\includegraphics[width=\linewidth]{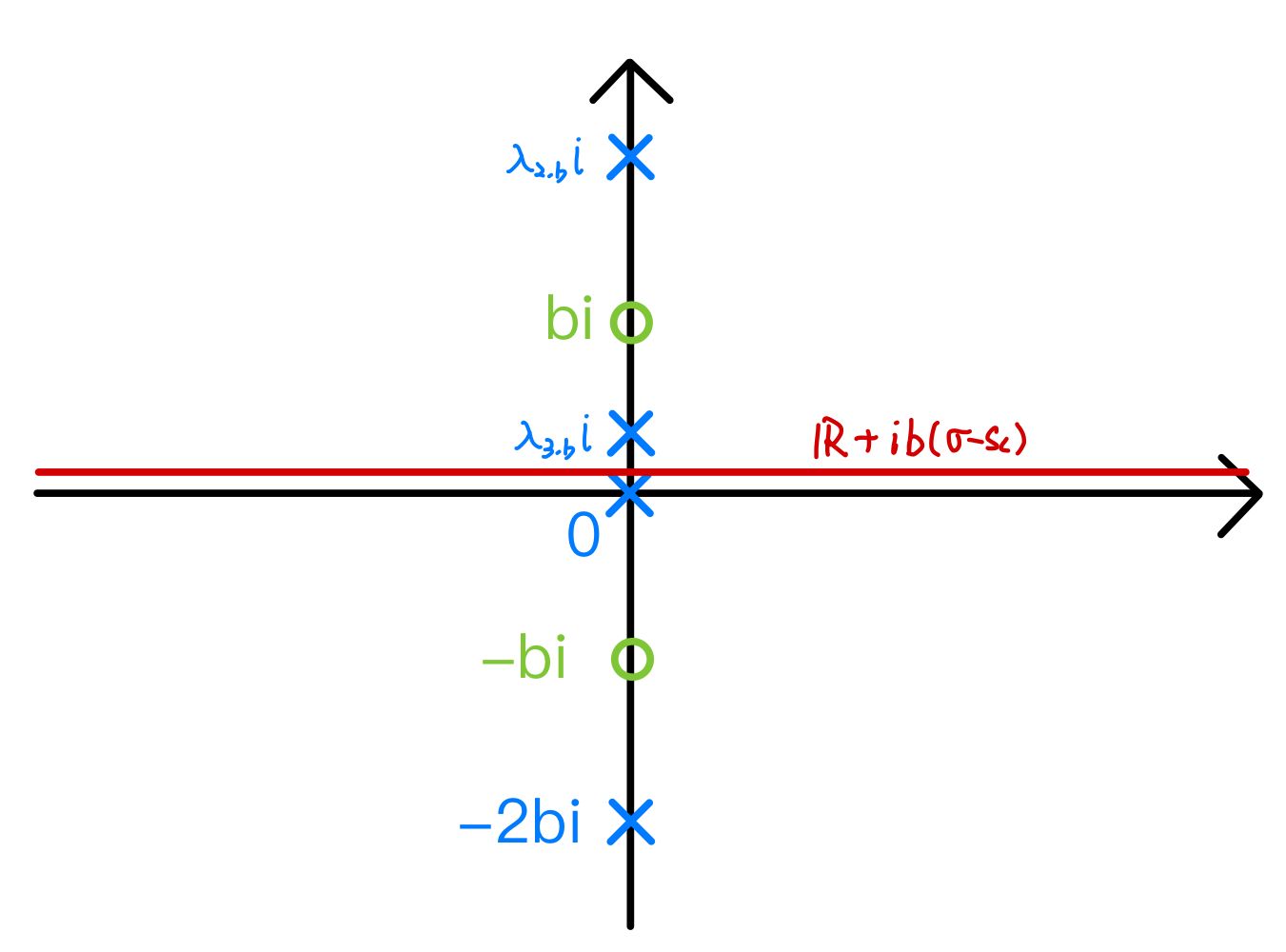}
			\caption{$\sigma(\calH_b \big|_{(\dot H^\sigma)^2})$}
		\end{subfigure}
		\caption{Spectrum of $\calH_0$ and $\calH_b$ near the origin: red line for $\sigma_{\rm ess}$, blue cross for eigenpairs in radial class, and green circle for eigenpairs in first spherical classes.}
		\label{fig:coffee}
\end{figure}

In this work, we aim to prove non-existence of unstable eigenvalues of $\calH_b$ away from the origin in $d \le 10$. As a byproduct of our analysis, we will also show non-existence of eigenvalues or resonances for $\calH_0$ away from the origin, which was unknown for $2 \le d \le 10$. In other words, 
our purpose is to verify Figure \ref{fig:coffee}(B) for $\Im \l < b(\sigma - s_c)$ and Figure \ref{fig:coffee}(A).

\mbox{}

\noindent\textit{Characterization of spectrum for Schr\"odinger-type operators.} In complement of the discussion of mode stability of self-similar profiles in \cite{Limodestability1}, we hereby emphasize the other aspect of our spectral problem: the Schr\"odinger-type linearized operator. 

The standard example is linearized operator around ground states of NLS of the form \eqref{eqdefH0}. Compared with linearized operator from semilinear parabolic equations or wave-type equations, the matrix nature brings in non-self-adjointness and possibility of embedded eigenvalues and resonances. Luckily, the structure of matrix potential implies a useful anti-diagonalization: 
\[ U \calH_0 U^{-1} = \left(\begin{array}{cc} 0 & iL_+ \\ -iL_- & 0 \end{array} \right)\quad {\rm where} \quad U = \frac 12 \left(\begin{array}{cc} 1 & 1 \\ -i & i \end{array} \right), \]
with the linear operators
\be L_\pm = -\Delta + 1 - W_1 \mp W_2. \label{eqdefLpm} \ee
Characterizing its spectrum is a core step towards asymptotic stability of ground states. For this active direction of research, we refer interesting readers to the survey \cite{zbMATH05208209,zbMATH07439618} and recent works \cite{arXiv:2408.15427,collot2023asymptotic,zbMATH07895064,zbMATH07991635} with references therein. 

For spectrum characterization related to ground states of \eqref{eqNLS}, the 1D case was relatively clear thanks to the explicit formula of the soliton and spectrum of corresponding $L_\pm$. The case $p \ge 5$ is done in \cite{MR1334139} and \cite{MR2219305}, and $p=3$ follows from the integrability structure (see for example \cite[Proposition 3.1]{arXiv:2408.15427} and the references therein). We also mention the recent result \cite{arXiv:2503.02957} in excluding embedded eigenvalues for 1D ground states of NLS with general nonlinearities.

For higher dimensions, notably, characterization of spectrum for the 3D cubic case is completed through a sequence of works: \cite{zbMATH04071530} for existence and uniqueness of stable/unstable eigenpairs, \cite{MR2480603} and \cite{MR2214946} showing non-existence of endpoint resonances and internal mode via argument in \cite{MR1334139} and computer-assisted proof for verification of a gap property for $L_\pm$, and finally \cite{marzuola2010spectral} in excluding embedded eigenvalues by numerically verifying the coercivity of a special quadratic form (see \cite[Section 2.3]{beceanu2012critical} for a nice summary). Although the method and result can be generalized, unfortunately, the mass-critical case when $d \ge 2$ seems out of reach as discussed in \cite{marzuola2010spectral}.

\subsection{Main results} The main input of this paper is the following two theorems, addressing the characterization of spectrum for $\calH_0$ and high-energy spectrum of $\calH_b$ respectively. The latter result combined with \cite{Limodestability1} and the author's previous work \cite{li2023stability} concludes the asymptotics stability of $Q_b$, which answers the question from \cite{MR4250747} and \cite{MR2729284}.

\begin{theorem}[Characterization of spectrum for $\calH_0$]\label{thmspecH0}
    For $1 \le d \le 10$, consider $\calH_0$ as a closed operator on $(L^2(\RR^d))^2$, the following statements hold true.
    \begin{enumerate}
        \item  $\sigma_{\rm ess}(\calH_0) = (-\infty, -1] \cup [1, \infty)$, $\sigma_{\rm disc}(\calH_0) = \{ 0 \}$ with the generalized nullspace given as \eqref{eqnullspaceHL2};
        \item There are no embedded eigenvalues in $\sigma_{\rm ess}(\calH_0)$, and the edges $\pm 1$ are not resonances\footnote{Defined in Definition \ref{defres}.}. 
    \end{enumerate}
\end{theorem}


\begin{theorem}[Mode stability of $\calH_b$ for high-energy spectrum]\label{thmextspec}
Let $1 \le d \le 10$. For any $\delta > 0$, there exists $0 < s_{c,3}(\delta) \ll 1$ such that for any
$0 < s_c \le s_{c,3}(\delta)$, $b, Q_b$ constructed as in \cite{MR4250747}, and any $\sigma \in (s_c, s_c + b^{d+4} )$, the discrete spectrum of $\calH_b$ satisfies
\be
  \sigma_{dist}\left(\calH_b\big|_{(\dot H^\sigma)^2}\right) \cap \left\{ z \in \CC: \Im z < b(\sigma -s_c), |z| \ge \delta \right\} = \emptyset.
\ee
    
\end{theorem}

Notice that Theorem \ref{thmextspec} and the low-energy mode stability \cite[Theorem 1.1]{Limodestability1} yield the full mode stability \cite[Assumption 1.2]{li2023stability} for $1 \le d \le 10$, $0 < s_c \ll 1$ and $0 < \sigma - s_c \ll 1$. Thus \cite[Theorem 1.4, Theorem 1.5]{li2023stability} lead to the following theorem. 



 \begin{theorem}[Asymptotical stability for $s_c \ll 1$.] \label{thmasympstab} Let $1 \le d \le 10$, $0 < s_c \ll 1$ small enough and $Q_b$ with $b = b(s_c)$ be the self-similar profile constructed in \cite{MR4250747}.
  Then for $0 < \sigma - s_c \ll 1$, the following statements hold:
  \begin{enumerate}
      \item Asymptotic stability in $\dot H^{\sigma}$: There exists $\epsilon_1 \ll 1$ such that for all $v_0 \in B_{\epsilon_1}^{\dot H^\sigma}$, there uniquely exists $(\l_0, x_0, \theta_0) \in \RR \times \RR^d \times \RR$ with 
  \be
    |\l_0 - 1| + |x_0| + |\theta_0| \lesssim \| v_0 \|_{\dot{H}^\sigma} \label{eqbddl0x0t0}
  \ee
  such that the initial data 
  \[ u_0 = Q_b + v_0 \]
  generates a solution blowing up at $T = \frac{\l_0^2}{2b}$ satisfying
  \bea 
      u(t, x) = \frac{1}{(\l_0^2 - 2b t)^{\frac{1}{p-1}}}  (Q_b + \e)\left(t, \frac{x-x_0}{\sqrt{\l_0^2-2bt}} \right)  e^{-i\left[\frac{\ln(\frac{1}{2b} - \l_0^{-2}t )}{2b} + \theta_0 \right]}, \label{eqss3}\\
        \| \e(t) \|_{\dot H^\sigma_x} \lesssim (1-2b \l_0^{-2} t)^{\frac{\sigma - s_c}{2}}.\label{eqss4}
   \eea
      \item Asymptotic stability in $H^1$: There exists an open set of initial data $\calO \subset H^1$ such that its solution to \eqref{eqNLS} blows up at $T = \frac{\l_0^2}{2b}$ and satisfies the decomposition \ref{eqss3} with $(\l_0, x_0, \theta_0) \in \RR \times \RR^d \times \RR$ bounded as \eqref{eqbddl0x0t0}. Moreover, we have the decay of perturbation as 
      \be \|\e(t)\|_{\dot{H}^\sigma \cap \dot H^1}\lesssim (T-t)^{\frac{\sigma - s_c}{2}},
	\label{eqbdd1H1} 
    \ee
    and 
    there exists $u_*\in \dot{H}^\sigma \cap \dot H^1$ and $u^* \in H^{\tilde \sigma}$ for every $0 \le \tilde \sigma < s_c$ such that 
		\bea
		u(t) - \frac{1}{\l(t)^{\frac{2}{p-1}}}Q_b \left( \frac{x}{\l(t)}\right)e^{i\tau(t)}  \to u_* && \mathrm{in}\,\,\dot{H}^\sigma\cap \dot H^1 \quad as\,\, t\to T,
		\label{eqbdd3H1} \\
        u(t) \to u^* && \mathrm{in}\,\,H^{\tilde \sigma} \quad \mathrm{as}\,\, t \to T. \label{eqL2profH1}
		\eea 
        The critical norms blow up as 
        \be \| u(t) \|_{L^{p_c}} \sim |\log(T-t)|^{\frac{1}{p_c}} (1+o(1)),\quad \| u(t)\|_{\dot{H}^{s_c}} \sim |\log(T-t)|^{\frac 12}(1+o(1)).\label{eqcriticalnorm} \ee
        when $t \to T$,
        where $p_c = \frac{2d}{d-2s_c}$.
  \end{enumerate}
 \end{theorem}

\mbox{}

\noindent\textit{Comments of Theorem \ref{thmspecH0} and Theorem \ref{thmextspec}.}

\mbox{}

\noindent\textit{1. Linear Liouville argument.}
The overall strategy we applied is called \textit{linear Liouville argument}, originated from Martel-Merle \cite{MR1826966,MR1753061} on gKdV soliton stability.
We consider eigenfunction as stationary solution of a linear evolutionary equation, so as to apply modulation argument and nonlinear dynamical control to prove its rigidity.\footnote{Although it is possible to avoid the evolutionary formulation, we prefer this dynamical way because otherwise it will sacrifice the clarity of modulation and insights for usage of energy/Virial identities, especially in the bifurcated problem.} The main edge of linear Liouville argument is to reduce the non-self-adjoint spectral problem to verifying coercivity of self-adjoint operators from some quadratic forms. Therefore, we expect this idea might be helpful in identifying the spectrum of other non-self-adjoint linearized operators around stationary states or their bifurcation. 

The natural nonlinear controlling laws for NLS are energy and Virial identities, and we exploit them in the linearized flow (Proposition \ref{propspecprop}). We stress that the coercivity of Virial commutator is proven with numerical help in \cite{MR2150386,MR2252148,MR3841347} for $1 \le d \le 10$. 
Originated from the log-log blowup analysis of Merle-Rapha\"el \cite{MR2150386}, this quantity also lies the foundation for the monotonicity formula used in the stability result of Merle-Rapha\"el-Szeftel \cite{MR2729284} (where it is referred to as \textit{spectral property} \cite[P. 1029]{MR2729284}).  Besides, our strategy can be viewed as a generalization of Marzuola-Simpson \cite{marzuola2010spectral} who also examined the same operator to characterize spectrum of $\calH_0$ for 3D cubic NLS, by clarifying its dynamical meaning and improving flexibility via introducing modulation.

We also mention Perelman's different argument in \cite{MR1852922} to prove high-energy mode stability for an operator like $\calH_b$. She first introduced an intermediate operator whose mode stability follows from being a compact perturbation of $\calH_0$, and mode stability of the target operator was derived by carefully evaluating the bifurcation of Jost function for high-energy. This argument is very general and robust, although dealing with high-energy Jost function uniformly can be quite complicated, especially for higher dimensions.

\mbox{}

\noindent{\textit{2. Adaption to the bifurcation operator $\calH_b$.}}
With the strategy and spectral property discussed above, the proof for $\calH_0$ case is straightforward, while the $\calH_b$ case requires more modification. Firstly, since the scaling generator naturally brings Virial commutator into the time-derivative of energy, the quantities to examine are two delicate (weighted and truncated) energies; secondly, we pick the sharp modulated directions via the algebraic relations related to $Q_b$, and the smallness for almost orthogonal condition restricts the potential eigenvalue $\l$ staying away from the origin. Finally, due to the weak decay $|W_{1, b}| + |W_{2, b}| \lesssim r^{-2}$ and lack of sharp Hardy inequality in the radial case for $d=2$, we need a delocalization estimate of possible eigenfunction through ODE analysis in the far exterior region.

\mbox{}

\noindent\textit{3. Extension of result.}
\begin{enumerate}
\item \textit{Higher dimensions:} As is discussed above and in the stability result Merle-Rapha\"el-Szeftel \cite{MR2729284}, the main theorems hold for other dimensions as long as the coercivity of Virial commutator (Proposition \ref{propspecprop} (2)) is true. 
    \item \textit{Linearization around ground states:} For $s_c \neq 0$, the generalized kernel of linearized operator has dimension $2d+2$, with the two radial generalized eigenfunction bifuracted to a pair of purely imaginary as $s_c > 0$ or real eigenvalues as $s_c < 0$ \cite[Theorem 2.6]{chang2008spectra}. Using these bifurcated eigenpairs, our method can lead to the characterization of spectrum for $d \le 10$ and $|s_c| \ll 1$. 
    \item \textit{Linearization around self-similar profile:} Our analysis can be adapted to the linear operator discussed in Perelman \cite{MR1852922} and its counterpart for $2 \le d \le 10$. However, we mention that extending the upper bound of $\sigma$ from $s_c + b^{d+4}$ to $O(b)$ as in the low-energy case requires new techniques or ideas, in view of the possible appearance of quantized eigenvalues (see the discussion in \cite[Comment 2 of Theorem 1.1]{Limodestability1}).
\end{enumerate}

\mbox{}

\subsection{Structure of the paper}

We will recall the asymptotics of $Q$, $Q_b$ and the known spectral properties of $\calH_0$ and $\calH_b$ in Section \ref{sec2}, including the crucial coercivity Proposition \ref{propspecprop}. Next, we prove Theorem \ref{thmspecH0} and Theorem \ref{thmextspec} in Section \ref{sec3}.

\subsection{Notations}
We use the following traditional notations: 
\begin{itemize}
\item $A \lesssim B$ means there exists $C$ such that $A \le CB$; $A \sim B$ means $A \lesssim B \lesssim A$.
    \item Japanese bracket $\la x \ra = (1 + |x|^2)^\frac 12$.
\item Ball of radius $\delta$ in Banach space $X$: $B_\delta^X = \{ x \in X: \| x \|_X < \delta \}$. The $X$ will be omitted if no ambiguity occurs, particularly when $X = \RR^d$.
\item Pauli matrix $\sigma_3 = \left( \begin{array}{cc}
    1 &  \\
     & -1
\end{array} \right)$. 
\item Weighted $\dot H^k$ space:
  $\| f \|_{\dot H^k_{\rho}} = \| \rho \nabla^k f \|_{L^2}$. Here $\nabla^k f = (\pa^{\vec \a} f)_{\substack{\vec \a \in \NN^d\\ \sum_i \a_i=k}}$ is understood as a vector-valued function.

\item Smooth radial cut-off function: Let 
\be \chi(r) = \left| \begin{array}{ll}
    1 & r \le 1 \\
    0 & r \ge \frac 32
\end{array}\right. \in  C^\infty_{c, rad}(\RR^d),\quad \chi_R = \chi(R^{-1}\cdot). \label{eqdefchiR} \ee
\end{itemize}

Besides, we have the following two conventions for notational simplicity:
\begin{itemize}
    \item For parameters $s_c$ and $b$: We assume throughout this paper that $0 < s_c \ll 1$ so that $b = b(s_c, d)$ and $Q_b$ from Proposition \ref{propQbasymp} exists; and we always denote $b = b(s_c, d)$ from it, and correspondingly $Q_b$, $W_{1, b}$, $W_{2, b}$ and $\calH_b$. 
    We will also usually omit the dependence of $d$ for every constant. 
    \item For vector-valued functions: the eigenfunctions of the matrix operator $\calH_0$ or $\calH_b$ are vector-valued, but we will keep using scalar Sobolev space if the two components are treated in that same space. For example, $Y_0 \in L^2$ indicates $Y_0 \in (L^2)^2$.
\end{itemize}

\mbox{}

\textbf{Acknowledgments.} The author is supported by the ERC advanced grant SWAT and ERC starting grant project FloWAS (Grant agreement No.101117820). He sincerely thanks his supervisor Pierre Rapha\"el for his constant encouragement and support. The author is also grateful to Charles Collot, Yvan Martel, and Cl\'ement Mouhot for valuable suggestions.

\section{Preliminary spectral properties}
\label{sec2}

\subsection{Linearization around ground states}

\mbox{}

\underline{Generalized kernel of $\calH_0$ and spectral properties.}
 
Let $Q$ be the mass-critical ground state \eqref{eqgroundstatemasscritical}. Define
\be Q_1 = \Lambda_0 Q, \quad Q_2 = \Lambda_0 Q_1, \quad L_+^{-1} (|x|^2 Q) = \rho. 
\label{eqdefQ1Q2rho} \ee
Here $L_+^{-1}$ is bounded on $L^2_{rad}$ \cite[Lemma 2.1]{chang2008spectra}. 
We also compute using $L_+ Q_1 = - 2 Q$ (see \cite[Section 2.1]{chang2008spectra}) that
\be
 \left( \rho, Q\right)_{L^2(\RR^d)} = \left(|x|^2 Q, -\frac 12\Lambda_0 Q\right)_{L^2(\RR^d)} = \frac{1}{2}\| xQ\|_{L^2(\RR^d)}^2,
 \label{eqrhobQ}
\ee

Now define the eigenmodes 
\be
\left| \begin{array}{l}
\xi_0 = i \left(\begin{array}{c} Q \\ -Q \end{array}\right), \quad
\xi_1 = \frac 12 \left(\begin{array}{c} Q_1 \\ Q_1 \end{array}\right),\\
\xi_2 = -\frac{i}{8} \left(\begin{array}{c} |x|^2 Q \\ -|x|^2 Q \end{array}\right), \quad
\xi_3 = \frac 18 \left(\begin{array}{c} \rho \\ \rho \end{array}\right). \\
\zeta_{0, j} = \left(\begin{array}{c} \pa_j Q \\ \pa_j Q \end{array}\right) \quad \zeta_{1, j} = -\frac i2 \left(\begin{array}{c} x_j Q \\ - x_j Q \end{array}\right),\quad 1 \le j \le d.
\end{array}\right. \label{eqdefeigenmodeH0}
\ee
Then we recall from \cite{MR0783974} or \cite[Section 2.1]{chang2008spectra} of the algebraic relations
\bee \calH_0 \xi_0 = 0, && \calH_0 \xi_{k} = -i\xi_k,\quad k = 1, 2, 3; \\ 
 \calH_0 \zeta_{0, j} = 0, && \calH_0 \zeta_1 = -i \zeta_0,\quad 1 \le j \le d.
\eee
and that the generalized nullspace of $\calH_0$ is generated by these vectors
\be N_g (\calH_0) = {\rm span} \{ \xi_0, \xi_1, \xi_2, \xi_3, \zeta_{0, j}, \zeta_{1, j}  \}_{1 \le j \le d}. \label{eqnullspaceHL2}\ee

The adjoint operator of $\calH_0$ in $L^2$ can be shown to be 
\be \calH_0^* = \sigma_3 \calH_0 \sigma_3, \ee
and therefore
\be N_g (\calH_0^*) = \sigma_3  N_g (\calH_0). \label{eqnullspaceH*L2}
\ee

\mbox{}

Next, we define the scalar Schr\"odinger operators
\be L_1 =  [L_+, \Lambda_0] =-\Delta+\frac{2}{d}\left(\frac{4}{d}+1\right) x \cdot \nabla Q Q^{\frac 4d - 1}. \quad L_2 =  [L_-, \Lambda_0] = -\Delta+\frac{2}{d} x \cdot \nabla Q Q^{\frac 4d - 1}. \label{eqdefL1L2}
\ee
Then $L_1, L_2$ and $L_\pm$ enjoy the following properties.
\begin{proposition}[Coercivity of energy and Virial commutator] \label{propspecprop}Let $u, v \in H^1(\RR^d)$.
\begin{enumerate}
    \item  For $d \ge 1$, suppose
    \be
      (u, Q)_{L^2} =  (u, xQ)_{L^2} =  (u, |x|^2 Q)_{L^2} =  (w, Q_1)_{L^2} =  (w, \rho)_{L^2} = 0,
      \label{eqortho1}
    \ee
    or 
    \be
      (u, Q)_{L^2} =  (u, xQ)_{L^2} =  (u, Q_1)_{L^2} =  (w, Q_1)_{L^2} = (w, Q_2)_{L^2} =  (w, \nabla Q)_{L^2} = 0, \label{eqortho2}
    \ee
     we have
    \be (L_+ u, u) \sim \| u \|_{H^1}^2,\quad (L_- w, w) \sim \| w \|_{H^1}^2. \label{eqcoerLpm} \ee 
    \item  Let $1 \le d \le 10$. Suppose $u, w$ satisfies \eqref{eqortho2}, we have
     \be (L_1 u, u) \sim \| u \|_{\dot H^1(\RR^d)}^2 + \| \la r \ra^{-\mu_d} u   \|_{L^2(\RR^d)}^2 ,\quad (L_2 w, w) \sim \| w \|_{\dot H^1(\RR^d)}^2 + \| \la r \ra^{-\mu_d} w   \|_{L^2(\RR^d)}^2. \label{eqcoerL12} \ee
     where $\mu_d = 1$ for $ d = 1$ or $d \ge 3$, and $\mu_2 = \frac{11}{10}$. Moreover, if $d=2$ but $u \in \dot H^1$ satisfies $\int_{|x| = r} u dz = 0$ for almost every $r > 0$, then $\mu_2$ can be taken as $1$.
\end{enumerate}
Moreover, there exists $\delta_d > 0$ such that \eqref{eqortho2} can be replaced by
\be
\begin{split}
&\max\left\{ |(u, Q)_{L^2}|, |(u, Q_1)_{L^2}|, |(u, xQ)_{L^2}|, |(w, Q_1)_{L^2}|, |(w, Q_2)_{L^2}|, |(w, \nabla Q)_{L^2}|\right\} \\
 \le& \delta_d \left(  \| \la r \ra^{-\mu_d} u   \|_{L^2(\RR^d)} +  \| \la r \ra^{-\mu_d} w   \|_{L^2(\RR^d)} \right)
 \end{split}
\label{eqalmostortho2}
\ee
with the above conclusions \eqref{eqcoerLpm} and \eqref{eqcoerL12} still true. 
\end{proposition}

This is a slight reformulation of results in  \cite{MR0783974} (or \cite{chang2008spectra}) and \cite{MR2150386,MR2252148,MR3841347} respectively for (1) and (2), and the proof is given in Appendix \ref{appA1}. We emphasize that (2) is proven with numerical help. 

Finally, we record the standard decay estimates for functions related to $Q$. 
\begin{lemma}For $d \ge 1$ and $x \in \RR^d$, 
 \bea
 Q(x) \sim \la x \ra^{-\frac{d-1}{2}} e^{-|x|};\quad  \left| \nabla^n Q(x) \right| \lesssim_n \la x \ra^{-\frac{d-1}{2}} e^{-|x|},\quad \forall\,\,n\ge 1. \label{eqsolitondecay} \\
 |\nabla^n f(x)|   \lesssim_n e^{-\frac 45 |x|},\qquad  \forall\,\, n \ge 0, \,\, f \in \{ Q_1, |x|^2 Q, Q_2, \rho \}.\label{eqdecayeigen} \eea
\end{lemma}
\begin{proof}
    The asymptotics of $Q$ \eqref{eqsolitondecay} can be found in, for example, \cite[Lemma 2.1]{Limodestability1}. Then the estimates for $Q_1, |x|^2 Q$ and $Q_2$ in \eqref{eqdecayeigen} easily follow from \eqref{eqsolitondecay}. 
    
    For $\rho$, we can first show $|\pa_r^n\rho| \lesssim \la x \ra^{3-\frac{d-1}{2}}e^{-|x|}$ for $n = 0, 1$ using the boundedness of $L_+^{-1}$ (see for example \cite[Lemma 3.1 (3.7)]{Limodestability1}). Then we obtain $n\ge 2$ decay estimates using the equation $L_+ \rho = |x|^2 Q$: firstly $\rho \in C^\infty(\RR^d)$ from elliptic regularity (which implies \eqref{eqdecayeigen} for $|x| \le 1$), and the decay for $|x| \ge 1$ can be derived inductively from
    $\pa_r^{n+2} \rho = \pa_r^n \left( ( 1- \frac{d-1}{r} - W_1 - W_2 ) \rho + r^2 Q\right)$ for $n \ge 0$ through differentiating the equation. 
\end{proof}

\subsection{Linearization around self-similar profile} 

\mbox{}

\underline{Existence and Asymptotics of $Q_b$.}

The existence of admissible $Q_b$ solving \eqref{eqselfsimilar} and its asymptotics near infinity were first obtained in \cite{MR4250747}. Here we recall the more refined asymptotics from \cite[Proposition 2.4, Proposition 4.13]{Limodestability1}.

\begin{proposition}[Existence and asymptotics of $Q_b$, {\cite{MR4250747,Limodestability1}}]\label{propQbasymp}
For any $d \ge 1$ and $0 < s_c \le s_c(d) \ll 1$, there exists $b = b(s_c, d) > 0$ with
\be s_c \sim b^{-1} e^{-\frac{\pi}{b}}, \label{eqbasymp}\ee
and $Q_b \in C^\infty(\RR^d) \cap \dot H^1(\RR^d)$ solving \eqref{eqselfsimilar} such that for $k \ge 0$, 
\begin{align}
 |\pa_r^k Q_b(r)| &\lesssim_k \left| \begin{array}{ll} (br)^{-k} r^{-\frac {d}2+s_c} b^{-\frac 12} 
e^{-\frac{\pi}{2b}} ,& r \ge 4b^{-1}; \\
r^{-\frac{d-1}{2}} b^{-\frac 16}e^{-\frac{\pi}{2b} + S_b(r)} \left\la b^{-\frac 23} (2-br)\right\ra^{-\frac 14}, & r \in [b^{-\frac 12}, 4b^{-1}]; \\
\la r \ra^{-\frac{d-1}{2}} e^{-r}, & r\in[0, b^{-\frac 12}];
\end{array}\right.\label{eqQbasymp1}\\
 |Q_b(r)| &\sim \left| \begin{array}{ll}  r^{-\frac {d}2+s_c} b^{-\frac 12} 
e^{-\frac{\pi}{2b}} ,& r \ge 4b^{-1}; \\
r^{-\frac{d-1}{2}} b^{-\frac 16}e^{-\frac{\pi}{2b} + S_b(r)} \left\la b^{-\frac 23} (2-br)\right\ra^{-\frac 14}, & r \in [b^{-\frac 12}, 4b^{-1}]; \\
\la r \ra^{-\frac{d-1}{2}} e^{-r}, & r\in[0, b^{-\frac 12}];
\end{array}\right.\label{eqQbasymp2}\\
  |\pa_r^k (\Re P_b-  Q)(r)| &\lesssim_k b^\frac 13 \la r \ra^{-\frac{d-1}{2}} e^{-r}, \quad r \le b^{-\frac 13}; \label{eqQbasympint1}\\
  |\pa_r^k \Im P_b(r)| &\lesssim_k  b s_c \la r \ra^{-\frac{d-1}{2}} e^{r}, \quad r \le b^{-\frac 13}. \label{eqQbasympint2}
\end{align}
  where $Q$ is the mass-critical ground state, $S_b(r) = \int_{\min\{ r, \frac 2b \}}^{\frac 2b} \left( 1 - \frac{b^2 s^2}{4} \right)^\frac 12$.
\end{proposition}

From now on, we refer to $Q_b$ as the solution of \eqref{eqselfsimilar} from Proposition \ref{propQbasymp}, and $\calH_b$ be the corresponding linearized operator \eqref{eqdefcalH}.
As a corollary, we have the following pointwise difference estimate for the potentials.

\begin{corollary}\label{corodiffest} Under the condition of Proposition \ref{propQbasymp}, with $W_1, W_2$ from \eqref{eqdefW1W20} and $W_{1, b}, W_{2,b}$ from \eqref{eqdefW1bW2b}, we have for $n \ge 0$,
\be
  \left|\pa_r^n (W_{j,b} - W_j)\right| \lesssim_n \left| \begin{array}{ll}
       b^\frac 13 |Q_b|^{p-1}, &  r \le b^{-\frac 13}, \\
       \la br \ra^{-n} |Q_b|^{p-1}, & r \ge b^{-\frac 13},
  \end{array}\right. 
        \quad {\rm for}\,\,j = 1, 2. \label{eqWjbdiffest}
\ee   
\end{corollary}
\begin{proof}
For $r \le b^{-\frac 13}$, we compute as in  \cite[(C.6)-(C.7)]{Limodestability1}, using the Fa\'a di Bruno's formula, Lebnitz rule and the non-vanishing asymptotics of $|Q_b|$ \eqref{eqQbasymp2} and $Q$ \eqref{eqsolitondecay}. The case for $r \ge b^{-\frac 13}$ follows from bounding $\pa_r^n W_{j, b}$, $\pa_r^n W_j$ separately, which is similar and simpler. 
\end{proof}

\mbox{}

\underline{Known spectrum of $\calH_b$.} 

Inspired by the eigenmodes of $\calH_0$ \eqref{eqdefeigenmodeH0}, we define  
\be \label{eqdefxi01b}
\left| \begin{array}{l}
\xi_{0,b} = i \left(\begin{array}{c} Q_b \\ -\bar{Q}_b \end{array}\right), \quad
\xi_{1,b} = \frac 12 \left(\begin{array}{c} \Lambda Q_b \\ \overline{\Lambda Q_b} \end{array}\right), \\
\xi_{2,b} = -\frac{i}{8} \left(\begin{array}{c} |x|^2 Q_b \\ -|x|^2 \bar{Q}_b \end{array}\right), \quad
\eta_b = \left(\begin{array}{c} Q_b \\ \bar{Q}_b \end{array}\right), \\
\zeta_{0, j, b} = \left(\begin{array}{c} \pa_j Q_b \\ \overline{\pa_j Q_b} \end{array}\right),\quad \zeta_{1,j,b} = - \frac i2 \left(\begin{array}{c} x_j  Q_b \\ -x_j \bar{Q}_b \end{array}\right),\quad 1 \le j \le d. 
\end{array}\right.
\ee
Then by differentiating the self-similar profile equation \eqref{eqselfsimilar}, we obtain the algebraic relations
\be
\begin{split}
    \calH_b \xi_{0, b} = 0,\quad \calH_b \xi_{1,b} = -2bi\xi_{1,b} - i\xi_{0,b},\quad \calH_b \xi_{2,b} = -i\xi_{1,b} + 2bi\xi_{2,b} - i\frac{s_c}{2}\eta_b, \\
\calH_b \zeta_{0,j,b} = -i \zeta_{0, j, b},\quad \calH_b \zeta_{1, j, b} = - i \zeta_{0, j, b} + i b \zeta_{1, j, b}.
\end{split} \label{eqQbalgrel}
\ee
thus 
\bea
\calH_b \xi_{0, b} = 0,\quad (\calH_b + 2bi) \left( \xi_{0,b} + 2b\xi_{1,b} \right) = 0,\label{eqeigencalHb1} \\
(\calH_b - 2bi)\left( \xi_{0,b} - 2b\xi_{1,b} + 8b^2 \xi_{2,b} \right) = -i 4b^2 s_c \eta_b,\label{eqeigencalHb2}\\  
(\calH_b + bi)\zeta_{0, j, b} = 0,\quad (\calH_b - bi) \left( \zeta_{0,j,b} - 2b\zeta_{1,j,b} \right) = 0.\label{eqeigencalHb3}
\eea

\mbox{}

\underline{A priori estimates for eigenfunctions}

We state the following lemma regarding a priori estimates for eigenfunctions. The proofs are using tools and arguments from the linear theory of \cite{li2023stability} and \cite{MR4250747}, which we record in Appendix \ref{appB1}.

\begin{lemma}\label{lemapriorieigen} For $d \ge 1$, when $0 < s_c \le s_c^* \ll 1$ and $\sigma \in (s_c, s_c + b)$, the following statements hold.
    Suppose $Y_0 \in \dot H^\sigma$ solves $(\calH_b - \l) Y_0 = 0$ with eigenvalue $\l$ satisfying $\Im \l < b(\sigma - s_c)$. We have 
    \begin{enumerate}
        \item (Regularity and smoothing estimate) $Y_0 \in \cap_{n \ge 1} \dot H^n$, and  
        \be \| Y_0 \|_{\dot H^2} \lesssim b^{-1-d}  \left( \| Y_0 \|_{\dot H^1_{\la x \ra^{-1}}} + \| Y_0 \|_{L^2_{\la x \ra^{-2}}}  \right).  \label{eqeigensmooth} \ee
        \item (Delocalization of radial eigenfunction) If $d =2$, then we further have
        \be  
           \| r^{-1} Y_0 \|_{L^2(\RR^2 - B_{4b^{-1}})} \lesssim b^{-\frac 12} \| r^{-\frac{11}{10}} Y_0 \|_{L^2(\RR^2)} \label{eqdelocalrad}
        \ee
    \end{enumerate}
\end{lemma}

\section{Mode stability by Linear Liouville argument}\label{sec3}

In this section, we prove Theorem and \ref{thmspecH0} and Theorem \ref{thmextspec} via \textit{linear Liouville argument}. The spectral property Proposition \ref{propspecprop} will be a core input for both proofs. 

\subsection{Mode stability for $\calH_0$}
 
We first recall the definition of endpoint resonance from \cite{schlag2005dispersive}).
\begin{definition}[Resonance] \label{defres}
    We say that $E = \pm 1$ is a resonance of $\calH_0$ provided $\calH_0 f = \pm f$ has a solution $f \notin L^2(\RR^d)$ but with 
    \be
     f \in L^\infty(\RR) \quad {\rm if}\,\,d=1; \quad  \la x \ra^{-\frac 12 -\epsilon} f \in L^2(\RR^d)\quad \forall \epsilon > 0,\quad {\rm if}\,\, d \ge 2. \label{eqresonancedef}
    \ee
\end{definition}
And we have the following regularity lemma. 
\begin{lemma}[Regularity of resonance state \cite{MR2219305,MR2480603}]\label{lemresonancereg}
    Suppose $E = \pm 1$ is a resonance of $\calH_0$ in $\RR^d$ with $f$ being the resonance state. Then $\nabla f \in L^2$ and $f \in H^\infty_{loc}$.
\end{lemma}
\begin{proof}
    For $d=1$, we use a sharper characterization from \cite[Lemma 5.19]{MR2219305}. 
    For $d \ge 2$, the $\dot H^1$ boundedness follows the proof \cite[Lemma 16]{MR2480603} via energy estimate of the scalar version of eigenequation
    \[ L_- L_+ f_1 = f_1. \]
    where we write $f = \left( \begin{array}{c}  f_1 \\ f_2 \end{array}\right)$.
    Standard interior elliptic regularity for $\calH_0 f = \pm f$ will imply the $H^\infty_{loc}$ smoothness.
\end{proof}

\mbox{}

Now we are in place to prove Theorem \ref{thmspecH0}.

\begin{proof}[Proof of Theorem \ref{thmspecH0}]
  \underline{1. Formulation of problem.}
  
  From Weyl's criterion and the decay of potential $W_1$, $W_2$, we know $\sigma_{\rm ess}(\calH_0) = \sigma_{\rm ess}(\calH_0) =  (-\infty, -1] \cup [1, \infty)$ where $\calH_0 = (\Delta-1)\sigma_3$. With \eqref{eqnullspaceHL2} characterizing the eigenspace at $0$, the theorem reduces to 
  \begin{enumerate}
      \item For any $\l \in \CC - \RR$, there exists no non-trivial $L^2$ solution $Z_0$ of $\calH_0 Z_0 = \l Z_0$;
      \item For any $\l \in \RR - \{ \pm 1, 0\}$, there exists no non-trivial $L^2$ solution $Z_0$ of $\calH_0 Z_0 = \l Z_0$; and there exists no non-trivial solution $Z_0 \notin L^2$ satisfying \eqref{eqresonancedef} for $\l = \pm 1$ either.  
  \end{enumerate}

  Suppose there exists $Z_0 = \left( \begin{array}{c}
       Z_0^{(1)}  \\
        Z_0^{(2)}
  \end{array}
  \right) \in L^\infty \cup L^2$ solves $\calH_0 Z_0 = \l Z_0$ with $\l \neq 0$. Then we have 
  \[ Z_0 \perp_{(L^2)^2} N_g (\calH_0^*), \]
  namely 
  \be
     Z_0^{(1)} +  Z_0^{(2)} \perp_{L^2} Q, |x|^2 Q, xQ \quad  Z_0^{(1)} -  Z_0^{(2)} \perp_{L^2} Q_1, \rho, \nabla Q. \label{eqorthoZ0}
  \ee
  Here the inner product makes sense due to the good decay \eqref{eqdecayeigen}. 
  Moreover, consider 
  \[ i\pa_t Z + \calH_0 Z = 0,\quad Z\big|_{t=0} = Z_0, \]
  then the global solution is $Z(t) = e^{i\l t} Z_0$ for $t \in \RR$, which also satisfies \eqref{eqorthoZ0}. 

  Let 
  $$ X(t) = \left( \begin{array}{c}
       u(t) \\
       w(t) 
  \end{array} \right) =: \left( \begin{array}{c}
       \frac 12 \left( Z^{(1)}(t)+ Z^{(2)}(t) \right)\\
       -\frac i2 \left( Z^{(1)}(t)- Z^{(2)}(t)\right)
  \end{array} \right), $$
  then we have 
  \bea
 \left|\begin{array}{l}
 L_+ u = -i\l w \\
 L_- w = i\l u 
   \end{array} \right.,\quad 
   \left|\begin{array}{l}
        \pa_t u = L_- w \\
        \pa_t w = -L_+ u
   \end{array} \right., \label{eqeigeneqX} \\
  X(t) = X_0 e^{i\l t}  \label{eqexplicitX} \\
   u(t) \perp_{L^2}  Q, |x|^2 Q, xQ, \quad  w(t)  \perp_{L^2} Q_1, \rho, \nabla Q. \label{eqorthoXt}
  \eea
  
  Next, we prove (1) and (2) separately by showing such $X_0$ must be trivial. 

  \mbox{}

  \underline{2. Rigidity for $\Im \l \neq 0$.} 
  Since $u, w \in L^2$ and $W_1, W_2 \in H^\infty$ from \eqref{eqdecayeigen}, the eigenequation immediately implies 
  \[ w = (1-\Delta )^{-1} \left( (W_1 - W_2)w - i\l u  \right) \in H^2,\quad u = (1-\Delta )^{-1} \left( (W_1 + W_2)u + i\l w  \right) \in H^2. \]
  And inductively, we have $u, w \in H^\infty$. 
  Now consider 
  \[ E(t) := (L_+ u, u)_{L^2} + (L_- w, w)_{L^2}. \]
  From the evolution and eigen equation \eqref{eqeigeneqX} and self-adjointness of $L_\pm$, we have
  \bee \pa_t E(t) &=&  \left( L_+ \pa_t u, u\right) + \left( L_+  u,\pa_t u\right) + \left( L_- \pa_t w, w\right) + \left( L_- w,\pa_t w\right) \\
   &=& -(L_- w, L_+u) - (L_+u, L_-w) + (L_+ u, L_-w) + (L_-w, L_+ u) = 0.
   \eee
   On the other hand, the spectral property \eqref{eqcoerLpm} and the explicit form of $X(t)$ \eqref{eqexplicitX} indicates 
   \bee
    E(t) \sim \| u(t) \|_{H^1}^2 + \| w(t)\|_{H^1}^2 = e^{-\Im \l t} \left( \| u_0 \|_{H^1}^2 + \| w_0\|_{H^1}^2\right)
   \eee
   These two conditions forces $u_0 = w_0 = 0$. 

   \mbox{}

   \underline{3. Rigidity for $\Im \l = 0$.}

\mbox{}

   \textit{(i) Improvement of regularity for $X(t)$.}
   For $\l = \pm 1$, Lemma \ref{lemresonancereg} implies 
   \be \la x \ra^{-\frac 12 -} u, \la x \ra^{-\frac 12 -} w \in L^2,\quad u, w \in \dot H^1 \cap  H^\infty_{loc}. \label{eqreguw}  \ee
   For $\l \in \RR - \{ \pm 1, 0\}$, similar argument as Step 2 implies $u, w \in H^\infty$ which also satisfies \eqref{eqreguw}.
   
   
   \mbox{}

   \textit{(ii) Modulation analysis.}

   Let 
   \be
   \ut(t) =u(t) + \gamma(t) Q_1,\quad \wt(t) = w(t) + \upsilon(t)Q \label{eqdefuwt}
   \ee
   with 
   \[ \gamma(t) = -\frac{(u(t), Q_1)_{L^2}}{\| Q_1\|_{L^2}^2},\quad \upsilon(t) = -\frac{(w(t), Q_2)_{L^2}}{(Q, Q_2)_{L^2}} = \frac{(w(t), Q_2)_{L^2}}{\| Q_1\|_{L^2}^2}.
   \]
   Recall \eqref{eqexplicitX}, we have $\gamma(t) = \gamma(0) e^{i\l t}$, $\upsilon(t) = \upsilon(0) e^{i\l t}$. Hence 
   \be
    \ut(t) = \ut(0) e^{i\l t} , \wt(t) = \wt(0) e^{i\l t}\label{eqexplicitXt}
   \ee
   have the same regularity as \eqref{eqreguw}. 
   
   Together with \eqref{eqorthoXt} and $Q_1 \perp_{L^2} Q, xQ$, $Q \perp Q_1, \nabla Q$, we have
   \be
     \ut \perp_{L^2} Q, Q_1, xQ,\quad \wt \perp_{L^2} Q_1, Q_2, \nabla Q. \label{eqorthouwt}
   \ee
   We compute the evolutionary equation of $\ut$, $\wt$: 
   \be
 \left|\begin{array}{l}
        \pa_t \ut = L_- \wt +\a(t) Q_1 \\
        \pa_t \wt = -L_+ \ut + \beta(t) Q
   \end{array} \right. \label{eqevouwt}
   \ee
   where 
   \bee  \a(t) = \gamma'(t) = i\l \gamma(0) e^{i\l t} ,\quad \beta(t) =- 2\gamma(t) + \upsilon'(t) = e^{i\l t}\left( -2\gamma(0) + i\l \upsilon(0) \right).  
   \eee
   Recall the definition of $\gamma(0), \upsilon(0)$ above, we have 
   \be |\a(t)| \lesssim \| e^{-\frac 14|x|} u_0 \|_{L^2},\quad |\beta(t)| \lesssim \| e^{-\frac 14|x|} u_0 \|_{L^2} +  \| e^{-\frac 14|x|} w_0 \|_{L^2}. \label{eqabeta}\ee
   Besides,  $\ut(t), \wt(t)$ have the same regularity as $u(t), w(t)$ in \eqref{eqreguw}. 

\mbox{}

   \textit{(iii) Truncated Virial identity and rigidity of $\ut, \wt$.}
   
Recall the smooth cutoff $\chi_R$ for $R > 0$ from \eqref{eqdefchiR}, and define the truncated vector field $\Phi_{R}(x):=x \chi_{R}(x)$. Then we define the truncated scaling generator for $R > 0$ as
\be 
\Lambda_{R}:=\frac{1}{2}\left(\nabla \circ \Phi_{R}+\Phi_{R} \circ \nabla\right)=\chi_{R}\Lambda_0 +x \cdot \nabla \chi_{R},
\ee
and in particular $\Lambda_R$ is also anti-symmetric.

Now we define the truncated Virial identity
\be
 I(t) := \Re \int_{\RR^d} \Phi_R \cdot \left(-\nabla \ut(t) \cdot  \overline{\wt(t)} + \nabla \wt(t) \cdot \overline{\ut(t)}\right) dx = -2\Re (\Lambda_R \ut, \wt)_{L^2}
\ee
It is well-defined since $\ut, \wt \in H^\infty_{\rm loc}$. By \eqref{eqexplicitXt}, it is indeed conserved
\be I(t) = I(0),\quad \forall t \in \RR. \label{eqItconserve}\ee

Compute with evolution equation for $\ut, \wt$ \eqref{eqevouwt} 
\bea
\pa_t I &=& -2\Re (\Lambda_R \pa_t \ut, \wt)_{L^2} -2 \Re (\Lambda_R \ut, \pa_t \wt)_{L^2}\nonumber \\
&=& -2\Re (\Lambda_R (L_- \wt + \a Q_1), \wt)_{L^2} -2\Re (\Lambda_R \ut, (-L_+ \ut + \beta Q))_{L^2}\nonumber \\
  &=& \Re \left( [L_-, \Lambda_R] \wt, \wt \right)_{L^2} + \Re \left( [L_+, \Lambda_R] \ut, \ut \right)_{L^2}- \Re \left(2\a \Lambda_R Q_1, \wt  \right) + \Re \left(2\beta \Lambda_R Q, \ut  \right) \nonumber \\
  &=& \Re (\chi_R L_1 \ut, \ut)_{L^2} + \Re (\chi_R L_2 \wt, \wt)_{L^2} \label{eqVirialline1} \\
  &+&\Re \left( [-\Delta, \chi_R] x\cdot \nabla \ut, \ut \right)_{L^2} + \Re \left( [-\Delta, \chi_R] x\cdot \nabla \wt, \wt \right)_{L^2} \label{eqVirialline2} \\
  &-& \Re \left( 2\beta (1-\chi_R) Q_1, \ut \right)_{L^2} + \Re \left( 2\a (1-\chi_R) Q_2, \wt \right)_{L^2}\label{eqVirialline3}
\eea
where in the last line we used the orthogonality $(Q_2, \wt)_{L^2} = (Q_1, \ut)_{L^2} = 0$. Now we estimate each term involving $\ut$, and the estimate for $\wt$ terms are similar.  

\mbox{}

(a) For \eqref{eqVirialline1},
\bee
\Re(\chi_R L_1 \ut, \ut)_{L^2} = (L_1 \ut, \ut)_{L^2} + \Re((1-\chi_R) L_1 \ut, \ut)_{L^2} \gtrsim \| \ut \|_{\dot H^1} + \| \la x \ra^{-\mu_d} \ut \|_{L^2} - o_{R\to \infty}(1)
\eee
Here we used \eqref{eqcoerL12} from Proposition \ref{propspecprop} (2), and the smallness of residual follows the upper bound. Note that the smallness is independent of $t$ because of \eqref{eqexplicitX}.

(b) For \eqref{eqVirialline2}, we compute
\bee
  &&\Re \left( [-\Delta, \chi_R] x\cdot \nabla \ut, \ut \right)_{L^2} = \Re \left(  \left(-\Delta \chi_R  -2\nabla \chi_R \cdot \nabla \right) x\cdot \nabla \ut, \ut \right)_{L^2} \\
  &=& \Re \left( \Delta \chi_R  x\cdot \nabla \ut, \ut \right)_{L^2} + \Re \left(  x\cdot \nabla \ut, 2\nabla \chi_R \cdot \nabla \ut \right)_{L^2} \\
  &=& -\frac 12 \Re \left( \nabla \cdot (x\Delta \chi_R) \ut, \ut \right)_{L^2} + \Re \left(  x\cdot \nabla \ut, 2\nabla \chi_R \cdot \nabla \ut \right)_{L^2}
\eee
So 
\bee
  \left| \Re \left( [-\Delta, \chi_R] x\cdot \nabla \ut, \ut \right)_{L^2} \right| \lesssim  \||\cdot|^{-1} \ut \|_{L^2 (B_{2R} - B_R)} + \|\nabla \ut \|_{L^2(B_{2R} - B_R) }= o_R(1)
\eee
the vanishing w.r.t. $R$ comes from the regularity \eqref{eqreguw}. 

(c) For \eqref{eqVirialline3}, we use \eqref{eqabeta}, \eqref{eqexplicitXt}
\bee
  &&\left|\Re \left( 2\a (1-\chi_R) Q_1, \ut \right)_{L^2}\right|\\
  &\le& 2\left( \| e^{-\frac 14 |x|} u_0 \|_{L^2} + \| e^{-\frac 14 |x|} w_0 \|_{L^2} \right) \cdot \| e^{-\frac 14 |x|} \ut_0 \|_{L^2(\RR^d - B_R)} =o_R(1).
\eee

\mbox{}

To sum up, taking $R = R(u_0, w_0)$ large enough, we can obtain
\bee
 \pa_t I(t) \gtrsim \| \ut_0\|_{\dot H^1} + \| \la x \ra^{-\mu_d} \ut_0 \|_{L^2} +  \| \wt_0 \|_{\dot H^1} + \| \la x \ra^{-\mu_d} \wt_0 \|_{L^2} \ge 0. 
\eee
By conservation of $I(t)$ \eqref{eqItconserve}, we must have 
\be \ut_0 = \wt_0 = 0. \label{eqrigiduwt} \ee

\mbox{}

\textit{(iv) Rigidity of $u, w$.}

Finally, recall that $u_0 \perp_{L^2} |x|^2 Q, w_0 \perp_{L^2} \rho$, we use \eqref{eqdefuwt} and \eqref{eqrigiduwt} to compute
\bee
 0 &=& \left( \ut_0, |x|^2 Q \right)_{L^2} = \gamma(0) \left( Q_1, |x|^2 Q \right)_{L^2} = -\gamma (0) \| xQ \|_{L^2}^2 \\
 0 &=& \left( \wt_0, \rho \right)_{L^2} = \upsilon(0) \left( Q, \rho \right)_{L^2} = \frac 12 \upsilon(0) \| xQ \|_{L^2}^2
\eee
where we used $L_+ Q_1 = -2Q$ from the algebraic relation $\calH_0 \xi_1 = -i\xi_0$. 
Thus $\upsilon(0) = \gamma(0) = 0$ and hence 
\[ u_0 = w_0 = 0. \]
\end{proof}

\subsection{High energy mode stability for $\calH_b$}

We now prove Theorem \ref{thmextspec} in this subsection. 

    

\begin{proof}[Proof of Theorem \ref{thmextspec}]

It suffices to consider $\delta \in (0, \frac 12)$, and we first restrict $s_{c, 3}(\delta)$ small enough to satisfy Proposition \ref{propQbasymp} and Lemma \ref{lemapriorieigen}. Notice that $\Im \l \le b(\sigma - s_c) \le b^{4+d}$ with $\sigma$ in the range required.

Suppose $Y_0 \in \dot H^\sigma$ satisfying $(\calH_b - \l) Y_0 = 0$ for $\l \in \{ z \in \CC: \Im z < b(\sigma - s_c), |z| \ge \delta\}$, we will show $Y_0 = 0$ by linear Liouville argument when $b(d, s_c) \ll 1$ small enough, whose range determines $s_{c,3}(\delta)$. 

\mbox{}

\underline{1. Formulation of problem.}

\textit{(1) Evolutionary equation.}

Denote the eigenfunction above as $Y_0 = \left( \begin{array}{c}
       Y_0^{(1)}  \\
        Y_0^{(2)}
  \end{array}
  \right)$. It generates a global solution $Y(t) = e^{i\l t} Y_0$ for $t \in \RR$ to 
  \[ i\pa_t Y + \calH_b Y = 0,\quad Y\big|_{t=0} = Y_0. \]

  Let $$ X(t) = \left( \begin{array}{c}
       u(t) \\
       w(t) 
  \end{array} \right) =: \left( \begin{array}{c}
       \frac 12 \left(Y^{(1)}(t)+ Y^{(2)}(t)\right) \\
       -\frac i2 \left( Y^{(1)}(t)- Y^{(2)}(t)\right)
  \end{array} \right), $$
  then $(\calH_b - \l) Y_0 = 0$ transforms into
  \be
    \left|\begin{array}{l}
        L_{+, b} u = \left( -i\l + bs_c - b\Lambda_0 + N_b\right) w\\
        L_{-, b} w = \left( i\l -bs_c + b\Lambda_0 + N_b \right) u
   \end{array} \right. \label{eqeigenY3}
  \ee
  where the $L^2$ self-adjoint operators $L_{\pm, b}$ and related potentials are 
  \be
  L_{\pm, b} = -\Delta + 1 - W_{\pm, b},\quad W_{\pm, b} := W_{1, b} \pm \Re W_{2, b}, \quad N_b = \Im W_{2, b}.\label{eqdefLpmbWpmbNb}
  \ee
  The evolutionary equation becomes
  \be
   \left|\begin{array}{l}
        \pa_t u = L_{-, b} w + (bs_c - b\Lambda_0 - N_b)u \\
        \pa_t w = - L_{+, b} u + (bs_c - b\Lambda_0 + N_b)w
   \end{array} \right., \label{eqevoeqX2}
  \ee
  and 
  \be X(t) = X_0 e^{i\l t}.  \label{eqexplicitX2} \ee


  \mbox{}


\textit{(2) Almost orthogonality conditions.}

Recall the smooth cutoff $\chi_R$ for $R > 0$ from \eqref{eqdefchiR} and that 
$$\supp\chi_R' = B_{2R} - B_{R}.$$
Define the truncated profile 
\[  \Qtb = Q_b \chi_{R_b},\quad R_b = \frac 4b \]
and the counterpart of \eqref{eqdefxi01b} 
\[  \tilde \xi_{0, b} = i\left( \begin{array}{c} \Qtb \\ \bar \Qtb \end{array}\right)^{\top},\quad  \tilde \xi_{1, b} = \frac 12 \left( \begin{array}{c} (\Qtb)_1 \\ \overline{(\Qtb)_1} \end{array}\right)^{\top}. \]
Also we denote 
\[ (\Qtb)_k = \Lambda^k \Qtb,\quad (Q_b)_k = \Lambda^k Q_b,\quad {\rm for}\,\,k = 1, 2. \]

We claim the following almost orthogonality conditions
\be
  \left|\left(Y_0, \xi \right)_{L^2} \right| \lesssim |\l|^{-4} b^\frac 16 \| \la x \ra^{-2} Y_0 \|_{L^2},\quad \forall\, \xi \in  N_g (\calH_0^*),  \label{eqaoc1}
  \ee
  where $N_g (\calH_0^*)$ is from \eqref{eqnullspaceH*L2} and \eqref{eqnullspaceHL2}, and the refined counterparts for $\xi \in \{ \xi_{0, 0}, \xi_{1, 0}\}$ 
  \be 
  \begin{split}
  \left| (Y_0, \sigma_3 \tilde \xi_{0, b})_{L^2} \right| \lesssim |\l|^{-1} b^{-1} s_c^\frac 12 \| \la x \ra^{-2} Y_0 \|_{L^2}, \\
\left| (Y_0, \sigma_3 \tilde \xi_{1, b})_{L^2} \right| \lesssim |\l|^{-2}  b^{-1} s_c^\frac 12 \| \la x \ra^{-2} Y_0 \|_{L^2}.
\end{split}\label{eqaoc2}
\ee
In particular, the first condition \eqref{eqaoc1} in $(u, w)$ variables is written as 
\bea
 \left|(u, Q)_{L^2}\right| + \left|(u, |x|^2 Q)_{L^2}\right| + \left|(u, xQ)_{L^2}\right| \lesssim |\l|^{-4} b^\frac 16 \big(\| u \|_{L^2_{\la x \ra^{-2}} } +  \|  w \|_{L^2_{\la x \ra^{-2}}}\big), \label{eqaou} \\
\left|(w, Q_1)_{L^2}\right| + \left|(w, \rho)_{L^2}\right| + \left|(w, \nabla Q)_{L^2}\right| \lesssim |\l|^{-4} b^\frac 16 \big(\| u \|_{L^2_{\la x \ra^{-2}} } +  \|  w \|_{L^2_{\la x \ra^{-2}}}\big). \label{eqaow}
\eea

\mbox{}

Indeed, for \eqref{eqaoc1}, we notice that 
the $L^2$-adjoint operator of $\calH_b$ is  
\[
  \calH_b^* = \left( \begin{array}{cc} \Delta_b -1 &  \\ &  -\Delta_{-b} +1\end{array} \right) +ibs_c + \left( \begin{array}{cc} W_{1, b} & -W_{2, b}  \\ \overline{W_{2, b}} & -W_{1, b}   \end{array} \right)  = \sigma_3 (\calH_b  + 2ibs_c) \sigma_3. 
\]
Hence 
\bee
 \triangle \calH_b:= \calH_b^* - \calH_0^* = (ib\Lambda_0 + ibs_c) I_2 + \triangle V,\quad \triangle V := \left(\begin{array}{cc} W_{1, b}-W_1 & - W_{2, b} + W_2  \\ \overline{W_{2, b}}-W_2 & -W_{1, b}+W_1   \end{array} \right).
\eee
We compute
\bee
 (Y_0, \xi)_{L^2} &=& \l^{-4} (\calH_b^4 Y_0, \xi)_{L^2} = \l^{-4} \left(Y_0, (\calH_0^* + \triangle \calH_b)^4 \xi \right)_{L^2} \\
 &=&  \l^{-4} \left(Y_0, \left((\calH_0^* + \triangle \calH_b)^4 - (\calH_0^*)^4\right) \xi \right)_{L^2},\quad \forall \xi \in N_g (\calH_0^*).
\eee
From \eqref{eqWjbdiffest}, \eqref{eqQbasymp1}, \eqref{eqsolitondecay} and \eqref{eqdecayeigen}, we easily obtain 
$$|\left((\calH_0^* + \triangle \calH_b)^4 - (\calH_0^*)^4\right) \xi | \lesssim \left|\begin{array}{ll}
    b^\frac 13 e^{-\frac 45 r}, &  r\le b^{-\frac 13};\\
    e^{-\frac 45 r},  & r \ge b^{-\frac 13};
\end{array}\right. $$
and therefore \eqref{eqaoc1} follows.

\mbox{}

For \eqref{eqaoc2}, 
we compute using the eigenequation of $Y_0$ and $\xi_{0, b}$, $\xi_{1, b}$ that 
\bee
  (Y_0, \sigma_3 \tilde \xi_{0, b})_{L^2} &=& (\l + 2ibs_c)^{-1} (Y_0, (\calH_b^* -2ibsc) \sigma_3 \tilde \xi_{0, b})_{L^2} \\
  &=& (\l + 2ibs_c)^{-1} (Y_0, \sigma_3 [\calH_b, \chi_{R_b}] \xi_{0, b})_{L^2}\\
 (Y_0, \sigma_3  \tilde \xi_{1, b})_{L^2}
  &=& (\l - 2bi(1 - s_c))^{-1} (Y_0, (\calH_b^* - 2ibs_c + 2bi) \sigma_3  \tilde \xi_{1, b} )_{L^2} \\
  &=&(\l - 2bi(1 - s_c))^{-1} (Y_0, \sigma_3  [\calH_b \Lambda, \chi_{R_b}] \xi_{1, b}) - i \sigma_3 \tilde \xi_{0, b} )_{L^2}
  \eee
Then \eqref{eqaoc2} follows this computation, $|\l| \ge b^\frac 12 \gg b$, 
\[ [\calH_b, \chi_{R_b}] = \sigma_3 (\Delta \chi_{R_b} + 2\nabla \chi_{R_b} \cdot \nabla) + ib|x|\chi_{R_b}',\quad [\Lambda, \chi_{R_b}] =|x|\chi_{R_b}',
\]
and
\[  \| \pa_r^n Q_b\|_{L^2(B_{2R_b} - B_{R_b})}^2 \lesssim_n s_c,\quad \forall \, n \ge 0,  \]
from \eqref{eqQbasymp1}. 

\mbox{}

Besides, we also record the algebraic relations about $\xi_{0, b}$, $\xi_{1, b}$ in \eqref{eqQbalgrel} in their scalar form
\bea
  L_{-, b} \Re Q_b + (b\Lambda + N_b) \Im Q_b &=& 0 \label{eqeigenrelation1}\\
  L_{+, b} \Im Q_b - (b\Lambda - N_b) \Re Q_b &=& 0 \label{eqeigenrelation2}\\
  L_{+, b} \Re (Q_b)_1 + (b\Lambda - N_b) \Im (Q_b)_1 &=& -2\Re Q_b - 2b\Im (Q_b)_1 \label{eqeigenrelation3}\\
  L_{-, b} \Im (Q_b)_1 - (b\Lambda + N_b) \Re (Q_b)_1 &=& -2\Im Q_b + 2b\Re (Q_b)_1. \label{eqeigenrelation4}
\eea
They will be used later.

\mbox{}

\textit{(3) Modulation analysis.}

   Let 
   \be
   \left| \begin{array}{l}
        \ut(t) =u(t) + \gamma(t) \Re (\Qtb)_1 - \upsilon(t) \Im \Qtb, \\ \wt(t) = w(t) + \gamma(t) \Im (\Qtb)_1 + \upsilon(t) \Re \Qtb,
   \end{array}\right.
   \label{eqdefuwt2}
   \ee
   with $\gamma(t)$ and $\upsilon(t)$ such that 
   \be
    (\ut, \Re (\Qtb)_1)_{L^2} + (\wt, \Im (\Qtb)_1)_{L^2} =  (\ut, -\Im (\Qtb)_2)_{L^2} + (\wt, \Re (\Qtb)_2)_{L^2} = 0.\label{eqorthoutwt1}
   \ee
   We claim the existence of $\gamma(t), \upsilon(t)$ with the following bound
   \bea
   |\gamma(t)| + |\upsilon(t)| \lesssim \la \l \ra^{-1} \left(\| \la x \ra^{-2} u \|_{L^2} + \| \la x \ra^{-2} w \|_{L^2}\right),
   \label{eqgammadeltabdd} 
   \eea
   and the equivalence of norms
   \be
    \| \la x \ra^{-2} u \|_{L^2} + \| \la x \ra^{-2} w \|_{L^2} \sim \| \la x \ra^{-2} \ut \|_{L^2} + \| \la x \ra^{-2} \wt \|_{L^2}.\label{equwutwtequiv}
   \ee
   Thereafter, recalling \eqref{eqexplicitX2}, we have 
   \be \gamma(t) = \gamma(0) e^{i\l t},\quad \upsilon(t) = \upsilon(0) e^{i\l t} \quad \Rightarrow \quad \gamma'(t) = i \l \gamma(t), \quad \upsilon'(t) = i \l \upsilon(t).\label{eqexplicitgd}\ee
   In particular, 
   \be
    \ut(t) = \ut(0) e^{i\l t} , \quad \wt(t) = \wt(0) e^{i\l t} \label{eqexplicitXt2}
   \ee
   have the same regularity as $u, w$ from Lemma \ref{lemapriorieigen} (1).
   
   Indeed, we will only prove \eqref{eqgammadeltabdd} and the $\lesssim$ direction of \eqref{equwutwtequiv}, which will imply the $\gtrsim$ direction of \eqref{equwutwtequiv} from the definition \eqref{eqdefuwt2}. For \eqref{eqgammadeltabdd}, plugging the orthogonal conditions into the equation implies 
   \bea
    &&\left( \begin{array}{cc}
       \| (\Qtb)_1\|_{L^2}^2 & -2(\Im \Qtb, \Lambda_0 \Re(\Qtb))_{L^2} \\
       2 (\Im (\Qtb)_1, \Lambda_0 \Re (\Qtb)_1)_{L^2}  & -  \| (\Qtb)_1\|_{L^2}^2 
    \end{array}\right)
    \left( \begin{array}{c}
         \gamma  \\
         \upsilon 
    \end{array} \right)\nonumber
    \\
    &=&
    \left( \begin{array}{c}
         - (u,  \Re (\Qtb)_1 )_{L^2} - (w,  \Im (\Qtb)_1)_{L^2}  \\
          - (u,  -\Im (\Qtb)_2 )_{L^2} - (w,  \Re (\Qtb)_2)_{L^2}
    \end{array} \right).  \label{eqsbsbsb}
   \eea
From Proposition \ref{propQbasymp}, we see
   \be \| \la x \ra^2  \Im (\Qtb)_k \|_{L^2} \lesssim_k b,\quad \| \la x \ra^2  \Re (\Qtb)_k \|_{L^2} \lesssim_k 1,\quad  {\rm for}\,\,k \ge 0.\label{eqQtbest}
   \ee
   This easily implies that the coefficient matrix on LHS has a uniformly bounded inverse, and the RHS is bounded by $O\left(\| \la x \ra^{-2} u \|_{L^2} + \| \la x \ra^{-2} w \|_{L^2}\right)$. To conclude \eqref{eqgammadeltabdd}, it suffices to show $O(|\l|^{-1})$ smallness of RHS of \eqref{eqsbsbsb}, which follows the computation involving \eqref{eqeigenY3} and \eqref{eqQtbest}. We estimate the first component as 
   \bee
      &&i\l \left[(u,  \Re (\Qtb)_1 )_{L^2} + (w,  \Im (\Qtb)_1)_{L^2} \right]\\
      &=& (L_{-, b}w - (b\Lambda + N_b)u,  \Re (\Qtb)_1 )_{L^2} - (L_{+, b}u + (b\Lambda - N_b)w,  \Im (\Qtb)_1 )_{L^2} \\
      &=& -(u, (-b\Lambda -2bs_c + \bar N_b) \Re (\Qtb)_1 + L_{+, b} \Im (\Qtb)_1 )_{L^2} \\
      &+& (w, (b\Lambda + 2bs_c + \bar N_b) \Im (\Qtb)_1 + L_{-, b} \Re (\Qtb)_1 )_{L^2} \\
      &=& O\left(\| \la x \ra^{-2} u \|_{L^2} + \| \la x \ra^{-2} w \|_{L^2}\right)
   \eee
   and the second component can be bounded likewise. 
   
   For the $\lesssim$ direction of \eqref{equwutwtequiv}, we rewrite \eqref{eqaoc1} with $\xi = \xi_{2, 0}$ and $\xi_{3, 0}$ as
   \bee
     \left| (\ut, |x|^2 Q)_{L^2} - \gamma (\Re (\Qtb)_1,  |x|^2 Q)_{L^2} + \upsilon (\Im \Qtb,  |x|^2 Q)_{L^2} \right| \lesssim |\l|^{-4} b^\frac 16 \| \la x \ra^{-2} Y_0 \|_{L^2}, \\
      \left| (\wt, \rho)_{L^2} - \gamma (\Im (\Qtb)_1,  \rho)_{L^2} - \upsilon (\Re \Qtb,  \rho)_{L^2} \right| \lesssim |\l|^{-4} b^\frac 16 \| \la x \ra^{-2} Y_0 \|_{L^2}.
   \eee
From Proposition \ref{propQbasymp} and \eqref{eqrhobQ}, we obtain
   \bee
     (\Re (\Qtb)_1,  |x|^2 Q)_{L^2} = -\|xQ \|_{L^2}^2 + O(b^\frac 16),&& (\Im \Qtb,  |x|^2 Q)_{L^2} = O(b^\frac 16),\\
     (\Re \Qtb, \rho)_{L^2} = \frac 12 \|xQ \|_{L^2}^2 + O(b^\frac 16),&& (\Im (\Qtb)_1,  \rho)_{L^2} = O(b^\frac 16).
   \eee
   Hence with \eqref{eqgammadeltabdd} and $|x|^2 Q, \rho \in L^2_{\la x\ra^2}$ \eqref{eqdecayeigen}, we have (with $|\l| \ge \delta$)
   \bee
    |\gamma| \lesssim \| \ut \|_{L^2_{\la x \ra^{-2}}} + \delta^{-4} b^\frac 16 \| \la x \ra^{-2} Y_0 \|_{L^2} \\
    |\upsilon| \lesssim \| \wt \|_{L^2_{\la x \ra^{-2}}} +  \delta^{-4} b^\frac 16 \| \la x \ra^{-2} Y_0 \|_{L^2}.
   \eee
   Plugging it back to \eqref{eqdefuwt2} and taking $L^2_{\la x \ra^{-2}}$ norm yield
   \bee
     \| u\|_{L^2_{\la x \ra^{-2}}} + \| w\|_{L^2_{\la x \ra^{-2}}} \lesssim \| \ut \|_{L^2_{\la x \ra^{-2}}} +  \| \wt \|_{L^2_{\la x \ra^{-2}}} + \delta^{-4} b^\frac 16 \| \la x \ra^{-2} Y_0 \|_{L^2}. 
   \eee
Then $\lesssim$ direction of \eqref{equwutwtequiv} follows by absorbing the last term in RHS to the left when $b \ll \delta^{-24}$. 

 \mbox{}

   Next, we compute the evolutionary equation of $\ut$, $\wt$: 
   \be
 \left|\begin{array}{l}
        \pa_t \ut = L_{-, b} \wt + (bs_c - b\Lambda_0 - N_b)\ut + f(t) \\
        \pa_t \wt = - L_{+, b} \ut + (bs_c - b\Lambda_0 + N_b)\wt + g(t)
   \end{array} \right. \label{eqevouwt2}
   \ee
   where
\bee
     f(t) &=& -\gamma \left( L_{-, b} \Im (\Qtb)_1 - (b\Lambda + N_b) \Re (\Qtb)_1 \right) + \gamma' \Re (\Qtb)_1 \\
     &-& \upsilon \left( L_{-,b} \Re \Qtb - (b\Lambda + N_b) (-\Im \Qtb)\right) + \upsilon' (-\Im \Qtb)  \\
     &=&  (2\gamma - \upsilon')\Im \Qtb + (-2b\gamma + \gamma') \Re (\Qtb)_1 -\gamma R_1 -\upsilon R_2 \\
     g(t) &=&  \gamma \left( L_{+, b} \Re (\Qtb)_1 + (b\Lambda - N_b) \Im (\Qtb)_1 \right) + \gamma' \Im (\Qtb)_1 \\
     &+& \upsilon \left( L_{+,b} (-\Im \Qtb) + (b\Lambda - N_b) \Re \Qtb)\right) + \upsilon' \Re \Qtb \\
     &=& (-2\gamma + \upsilon') \Re \Qtb + ( -2b\gamma + \gamma') \Im (\Qtb)_1 + \gamma R_3 + \upsilon R_4. 
   \eee
   Here we applied \eqref{eqeigenrelation1}-\eqref{eqeigenrelation4} and the error terms $R_k$ for $1 \le k \le 4$ are given by
   \bee
   R_1 &=& 2b[\chi_{R_b}, \Lambda]\Re Q_b +  [L_{-, b}\Lambda, \chi_{R_b}]\Im Q_b - [b\Lambda^2 + N_b \Lambda, \chi_{R_b}]\Re Q_b,\\
   R_2 &=& [L_{-, b}, \chi_{R_b}] \Re Q_b + b[\Lambda, \chi_{R_b}] \Im Q_b \\
   R_3 &=& -2b[\chi_{R_b}, \Lambda]\Im Q_b + [L_{+, b} \Lambda, \chi_{R_b}] \Re Q_b + [b\Lambda^2 - N_b \Lambda, \chi_{R_b}]\Im Q_b \\
   R_4 &=& [L_{+, b}, \chi_{R_b}] (-\Im Q_b) + b[\Lambda, \chi_{R_b}] \Re Q_b.
   \eee
%
%
Note that $R_k$ for $1 \le k \le 4$ all involve commutator with $\chi_{R_b}$ thus are supported on $r \in \left[\frac 4b, \frac 8b\right]$. So we bound these residuals by the pointwise estimates \eqref{eqQbasymp1}, \eqref{eqWjbdiffest}, and combine \eqref{eqgammadeltabdd}, \eqref{equwutwtequiv}, \eqref{eqexplicitgd} to obtain
\be 
 \| f(t)\|_{H^2_{\la x \ra^2}} + \| g(t)\|_{H^2_{\la x \ra^2}}  \lesssim \| \tilde u \|_{L^2_{\la x \ra^{-2}}}  +  \| \tilde w \|_{L^2_{\la x \ra^{-2}}} \label{eqbddfg}
\ee
\mbox{}

Finally, we check that the almost 
orthogonality conditions for $\ut, \wt$ as
\bea
 \left|(\ut, Q)_{L^2}\right| + \left|(\ut, Q_1)_{L^2}\right| + \left|(\ut, xQ)_{L^2}\right| \lesssim |\l|^{-4} b^\frac 16 \big(\| \ut \|_{L^2_{\la x \ra^{-2}} } +  \|  \wt \|_{L^2_{\la x \ra^{-2}}}\big), \label{eqaout} \\
\left|(\wt, Q_1)_{L^2}\right| + \left|(\wt, Q_2)_{L^2}\right| + \left|(\wt, \nabla Q)_{L^2}\right| \lesssim |\l|^{-4} b^\frac 16 \big(\| \wt \|_{L^2_{\la x \ra^{-2}} } +  \|  \wt \|_{L^2_{\la x \ra^{-2}}}\big), \label{eqaowt} \\
\left| (\ut, \Re \Qtb)_{L^2} + (\wt, \Im \Qtb)_{L^2} \right| \lesssim |\l|^{-1} b^{-1} s_c^\frac 12 \big(\| \ut \|_{L^2_{\la x \ra^{-2}} } +  \|  \wt \|_{L^2_{\la x \ra^{-2}}}\big), \label{eqorthoutwt2} \\
\left| (\ut, \Im (\Qtb)_1)_{L^2} - (\wt, \Re (\Qtb)_1)_{L^2} \right|  \lesssim |\l|^{-2} b^{-1}s_c^\frac 12 \big(\| \ut \|_{L^2_{\la x \ra^{-2}} } +  \|  \wt \|_{L^2_{\la x \ra^{-2}}}\big).\label{eqorthoutwt3}
\eea
The first two inequalities \eqref{eqaout}-\eqref{eqaowt} follow from (almost) orthogonality of \eqref{eqorthoXt} (similar to the proof of \eqref{eqorthouwt} in Theorem \ref{thmspecH0}) and \eqref{eqorthoutwt1} plus boundedness of residuals. We compute
 \bee
  && {\rm LHS \,\,of\,\,\eqref{eqaout}} \\
  &\lesssim& |(\gamma \Re (\Qtb)_1 - \upsilon \Im \Qtb, Q)_{L^2}|  +   |(\gamma \Re (\Qtb)_1 - \upsilon \Im \Qtb, xQ)_{L^2}| \\
  &+& |(\ut, Q_1 - \Re (\Qtb)_1)_{L^2} - (\wt, \Im (\Qtb)_1)_{L^2} | +  |\l|^{-4} b^\frac 16 \big(\| \ut \|_{L^2_{\la x \ra^{-2}} } +  \|  \wt \|_{L^2_{\la x \ra^{-2}}}\big), \\
  &\lesssim& {\rm RHS \,\,of\,\,\eqref{eqaout}}.
 \eee
 Here for the second inequality we used $(Q_1, Q)_{L^2} = (Q_1, xQ)_{L^2} = 0$, \eqref{eqgammadeltabdd}, \eqref{eqQtbest} and 
\[ \sum_{k=0}^2 \left\| \la x \ra^2 \left(\Re (\Qtb)_k - Q_k \right) \right\|_{L^2} \lesssim b^\frac 16 \]
from Proposition \ref{propQbasymp}. The proof of \eqref{eqaowt} is similar. 
For \eqref{eqorthoutwt2}-\eqref{eqorthoutwt3}, notice that they can be obtained by replacing $u, w$ with $\ut, \wt$  in \eqref{eqaoc2}. To replace the $u, w$ on RHS; we apply \eqref{equwutwtequiv}; and for LHS, we view the modulation part as perturbation through
\bee
  \left| (\Re (\Qtb)_1, \Re  \Qtb)_{L^2} \right| + \left| (\Im (\Qtb)_1, \Im \Qtb)_{L^2} \right| &=& s_c \| (\Qtb)_1 \|_{L^2}^2 \lesssim s_c  \\
    \left| (\Re  \Qtb, \Im \Qtb)_{L^2} \right| +  \left| (\Re (\Qtb)_1, \Im (\Qtb)_1)_{L^2} \right| &\lesssim& b^{-2} s_c
\eee
from the anti-symmetry of $\Lambda_0 = \Lambda + s_c$ and the asymptotics in Proposition \ref{propQbasymp}.

\mbox{}

\underline{2. $\Im \l < -b^{4+d}$: energy estimate.}

Consider 
\[ E_b(t) = (L_{+, b} \ut, \ut)_{L^2} + (L_{-, b} \wt, \wt)_{L^2} - 2\Re(N_b \ut, \wt)_{L^2}, \]
and we can compute
\bea
 && e^{2bs_ct} \pa_t \left( e^{-2bs_ct} E_b\right)\nonumber \\
 &=& - b \left\{ ([L_{+, b}, \Lambda_0]\ut, \ut)_{L^2}+   ([L_{-, b}, \Lambda_0]\wt, \wt)_{L^2} - 2 \Re ([N_b, \Lambda_0]\ut, \wt)_{L^2} \right\}\label{eqenergyRHS1}\\
 &+& 2 \Re (\ut, L_{+, b}f - N_b g)_{L^2} +2 \Re (\wt, L_{-, b} g - N_b f)_{L^2} \label{eqenergyRHS2}
\eea
We claim that 
\bea
  \left| \eqref{eqenergyRHS2} \right| &\lesssim& \delta^{-2} b^{-1} s_c^\frac 12 \left(\|\ut \|_{L^2_{\la x \ra^{-2}}}^2 +  \|\wt \|_{L^2_{\la x \ra^{-2}}}^2 \right) 
  \label{eqresfgest} \\
   \eqref{eqenergyRHS1} &\lesssim& -b \left( \| \tilde u\|^2_{\dot H^1 \cap L^2_{\la x \ra^{-2}}} + \| \tilde w\|^2_{\dot H^1 \cap L^2_{\la x \ra^{-2}}}  \right). \label{eqenergy1coer}
\eea
Then $e^{-2bs_ct} E_b$ is monotonically decreasing when $b \ll 1$. Also recall \eqref{eqexplicitXt2} indicates that
\[ e^{2bs_ct}\pa_t \left( e^{-2bs_ct} E_b\right) = e^{2bs_ct} \pa_t \left( e^{(-2bs_c - 2 \Im \l) t } E_b(0) \right) = (-2bs_c - 2\Im \l) E_b.\]
With $\Im \l + bs_c < -b^{4+d} + bs_c < 0$, we have $E_b \le 0$. Further from Proposition \ref{propspecprop} (1) and almost orthogonal conditions \eqref{eqaout}-\eqref{eqaowt}, we know $E_b \ge 0$. Hence $E_b = 0$, $\ut = \wt = 0$, which leads to $Y_0 = 0$ by \eqref{equwutwtequiv}.  

\mbox{}

\textit{Proof of \eqref{eqresfgest}.} Similar to the computation of $f, g$, we apply \eqref{eqeigenrelation1}-\eqref{eqeigenrelation4} and single out the error terms localized away from the origins to get
\bee
  L_{+, b} f - N_b g 
  &=& (2\gamma - \upsilon')b\Re (\Qtb)_1 + (2b\gamma - \gamma')\left( b\Im (\Qtb)_2 + 2 \Re \Qtb + 2b \Im (\Qtb)_1  \right) \\
  &+& (2\gamma - \upsilon') R_5 + (2b\gamma - \gamma') R_6 - \gamma (L_{+, b} R_1 + N_b R_3) - \upsilon (L_{+, b} R_2 + N_b R_4) \\
  L_{-, b}g - N_b f&=& (2\gamma - \upsilon')b\Im (\Qtb)_1 + (2b\gamma - \gamma')\left(-b\Re (\Qtb)_2 + 2 \Im \Qtb - 2b \Re (\Qtb)_1  \right) \\
  &+& (2\gamma - \upsilon') R_7 + (2b\gamma - \gamma') R_8 + \gamma (L_{-, b} R_3 + N_b R_1) + \upsilon (L_{-, b} R_4 + N_b R_2)
\eee
where
\bee
  R_5 &=& [L_{+, b}, \chi_{R_b}]\Im Q_b - b[\Lambda, \chi_{R_b}]\Re Q_b \\
  R_6 &=& [\chi_{R_b}, \Lambda^2] b\Im Q_b +[\chi_{R_b}, \Lambda]2b \Im Q_b -  [L_{+, b}\Lambda, \chi_{R_b}]\Re Q_b + [N_b \Lambda, \chi_{R_b}] \Im Q_b \\
  R_7 &=& - [L_{-, b}, \chi_{R_b}]\Re Q_b - b[\Lambda, \chi_{R_b}]\Im Q_b \\
  R_8 &=& -[\chi_{R_b}, \Lambda^2] b\Re Q_b - [\chi_{R_b}, \Lambda] 2b\Re Q_b - [L_{-, b}\Lambda, \chi_{R_b}]\Im Q_b + [N_b \Lambda, \chi_{R_b}] \Re Q_b.
\eee
Now we compute 
\bee
  &&\Re (\ut, L_{+, b}f - N_b g)_{L^2} + \Re (\wt, L_{-, b} g - N_b f)_{L^2} \\
  &=& \Re \left\{ \overline{(2\gamma - \upsilon')}b\left[(\ut, \Re (\Qtb)_1)_{L^2} + (\wt, \Im (\Qtb)_1)_{L^2} \right]\right\} \\
  &+&  \Re  \Big\{ \overline{(2b\gamma - \gamma')}\big[(\ut, b \Im  (\Qtb)_2 + 2\Re \Qtb + 2b \Im (\Qtb)_1)_{L^2} \\
  &&+ (\wt, -b\Re (\Qtb)_2 + 2 \Im \Qtb - 2b \Re (\Qtb)_1)_{L^2} \big]\Big\} \\
  &+& \Re (\ut, (2\gamma - \upsilon') R_5 + (2b\gamma - \gamma') R_6 - \gamma (L_{+, b} R_1 + N_b R_3) - \upsilon (L_{+, b} R_2 + N_b R_4))_{L^2} \\
  &+& \Re (\wt, (2\gamma - \upsilon') R_7 + (2b\gamma - \gamma') R_8 + \gamma (L_{-, b} R_3 + N_b R_1) + \upsilon (L_{-, b} R_4 + N_b R_2))_{L^2}
\eee
Note that $R_k$ for $1 \le k \le 8$ are supported on $r \in \left[\frac 4b, \frac 8b\right]$.
We apply the boundedness of $\gamma, \upsilon$ \eqref{eqgammadeltabdd}, \eqref{equwutwtequiv} and \eqref{eqexplicitgd} for their derivatives, (almost) orthogonal conditions \eqref{eqorthoutwt1}, \eqref{eqorthoutwt2} and \eqref{eqorthoutwt3}, and the pointwise  bound \eqref{eqQbasymp1}-\eqref{eqQbasymp2} for $R_k$, $1 \le k \le 8$, to obtain the estimate \eqref{eqresfgest}.

\mbox{}

\textit{Proof of \eqref{eqenergy1coer}.}
Since $\calH_b$ preserves spherical harmonic decomposition (see \cite[(1.25)]{Limodestability1} or \cite[Chap. IV, Sec. 2]{MR0304972}) thanks to the radial symmetry of $Q_b$, we can assume $Y_0$ belongs to one spherical class. So we consider the following two cases.

(i) Suppose $d \neq 2$, or $d = 2$ with $Y_0$ satisfies $\int_{|x| = r} Y_0 dz = 0$ for almost every $r > 0$, namely it has no radial components in the spherical harmonics decomposition.  Then \eqref{eqaout}-\eqref{eqaowt} and Proposition \ref{propspecprop} yield when $|\l| \ge \delta$ and $b \ll \delta^{24}$,
\bee
  ([L_+, \Lambda_0] \ut, \ut)_{L^2} + ([L_-, \Lambda_0] \wt, \wt)_{L^2} \gtrsim \| \nabla \ut\|_{L^2}^2 + \| \la r \ra^{-1} \ut \|_{L^2}^2 +  \| \nabla \wt\|_{L^2}^2 + \| \la r \ra^{-1} \wt \|_{L^2}^2.
\eee
Also notice $L_{\pm, b} - L_\pm = (W_1 - W_{1, b}) \pm (W_2 - \Re W_{2, b})$ and the estimate from \eqref{eqWjbdiffest} and \eqref{eqQbasymp2}
\[ |x\cdot \nabla (W_1 - W_{1, b})| + |x\cdot \nabla (W_2 - \Re W_{2, b})| + |x\cdot \nabla \Im W_{2, b}| \lesssim b^\frac 16 \la r \ra^{-2}. \]
Thus we have the control for residuals
\bee
\left| ([L_{+, b} - L_+, \Lambda_0]\ut, \ut)_{L^2} \right|+ \left|  ([L_{-, b} - L_-, \Lambda_0]\wt, \wt)_{L^2}\right| + \left| ([N_b, \Lambda_0]\ut, \wt)_{L^2} \right| \\
\lesssim b^\frac 16 \left(  \| \la r \ra^{-1} \ut \|_{L^2}^2 +  \| \la r \ra^{-1} \wt \|_{L^2}^2 \right).
\eee
That concludes the coercivity \eqref{eqenergy1coer}. 

(ii) Suppose $d=2$ and $Y_0 \in L^2$ is radial. Similarly, the leading order is coercive within $L^2_{\la x \ra^{-\frac{11}{10}}}$ from Proposition \ref{propspecprop}. To control the residual, we refine the estimate of potential again from \eqref{eqWjbdiffest} and \eqref{eqQbasymp2}
\[ |x\cdot \nabla (W_1 - W_{1, b})| + |x\cdot \nabla (W_2 - \Re W_{2, b})| + |x\cdot \nabla \Im W_{2, b}| \lesssim b^\frac 16 \la r \ra^{-\frac{11}{5}} + b^{-2} s_c^{\frac{p-1}{2}} (1-\chi_{R_b})\la r \ra^{-2}, \]
and therefore \eqref{eqbasymp} and the delocalization estimate \eqref{eqdelocalrad} imply
\bee
&&\left| ([L_{+, b} - L_+, \Lambda_0]\ut, \ut)_{L^2} \right|+ \left|  ([L_{-, b} - L_-, \Lambda_0]\wt, \wt)_{L^2}\right| + \left| ([N_b, \Lambda_0]\ut, \wt)_{L^2} \right| \\
&\lesssim& b^\frac 16 \left(  \| \la r \ra^{-\frac{11}{10}} \ut \|_{L^2}^2 +  \| \la r \ra^{-\frac{11}{10}} \wt \|_{L^2}^2 \right)
\eee
when $b \ll 1$. 
That yields \eqref{eqenergy1coer} in this case. 

\mbox{}\\

\underline{3. $|\Im \l| \le b^{4+d}$: weighted energy estimate.} 

For notational simplicity, we denote 
\bee \mu = \frac{\Im \l}{b} + s_c \quad \Rightarrow \quad |\mu| \le 2 b^{3+d} \eee
when $s_c \ll 1$ such that $b \ll 1$.

Consider the truncated and weighted energy
\bee
E_{b, \mu} = (\tilde L_{+, b} \varrho^\mu \ut, \varrho^\mu \ut)_{L^2} +  (\tilde L_{-, b} \varrho^\mu \wt, \varrho^\mu \wt)_{L^2} - 2 \Re (\tilde N_b \varrho^\mu \ut, \varrho^\mu \ut)_{L^2}
\eee
where the weight function $\varrho$ and truncated operators $L_{\pm, b}, N_b$ (see \eqref{eqdefLpmbWpmbNb}) are
\bee
  \varrho = (1 + r^4)^{-\frac 14},
\quad
  \tilde L_{\pm, b} = -\Delta + (1 - W_{\pm, b}) \psi_b,\quad
  \tilde N_b = \Im W_{2, b} \psi_b,
\eee
and $\psi_b$ is a $C^3_c$ cutoff function defined below.  

\mbox{}

We require $\psi$ be taken such that $\psi(r)= \left| \begin{array}{ll} 
  1 & r \le 1, \\
  0 & r \ge 2.
  \end{array}\right. \in C^3_{c, rad}(B_2)$ and 
\bee
  -\psi'(r) \sim \min \left\{ r-1, 2-r \right\}^4,\quad |\psi''(r)| \lesssim  \min \left\{ r-1, 2-r \right\}^3,\quad r \in [1, 2].
\eee
This can be constructed using piecewise fourth-order polynomials on $r \in [1, 2]$. Let
\be
R = b^{-2}, \quad \psi_b = \psi(R^{-1}\cdot),\quad  \tilde \psi_b := -r\pa_r \psi_b.
\ee
 We claim
\be
  |\pa_r^2 \psi_b| \lesssim b^\frac 75 (\tilde \psi_b + r^{-\frac{11}{5}}).\label{eqpsibrr}
\ee
Notice that $\pa_r^2 \psi_b$ is supported in $r \in [R, 2R]$, by rescaling \eqref{eqpsibrr} is equivalent to
\be R^{-2} |\psi''| \lesssim b^\frac 75 (-\psi' + R^{-\frac{11}{5}}) \label{eqpsibrr1} \ee
We consider the following two cases. (i) When $r - 1 \in [0, R^{-\frac{3}{10}}]$ or $2 - R \in [0, R^{-\frac 3{10}}]$, we have
\bee
  |\psi''| \lesssim R^{-\frac{9}{10}} = b^{\frac{9}{5}} \le b^\frac 75 R^{-\frac 15}.
\eee
(ii) When $r \in [1 +  R^{-\frac{3}{10}}, 2 - R^{-\frac{3}{10}}]$, we have $|\psi''| \lesssim 1$ and
\bee
b^\frac 75 (-\psi') \gtrsim b^\frac 75 (R^{-\frac 3{10}})^4 \ge R^{-2} \gtrsim R^{-2} |\psi''|.
\eee
These two estimates imply \eqref{eqpsibrr1} and thus \eqref{eqpsibrr}. 

\mbox{}

Similar to Step 2, we apply the evolution equation \eqref{eqevouwt2} to compute
\[  \frac 12 \pa_t E_{b,\mu} = bs_c E_{b, \mu} + \sum_{j = 1}^4 I_j \]
 where 
 \bee
 I_1 &=&-b \Big\{ \Re\left(\tilde L_{+, b} \varrho^\mu \ut, [\varrho^\mu,\Lambda_0]\ut \right)_{L^2} + \Re\left(\tilde L_{-, b} \varrho^\mu \wt, [\varrho^\mu,\Lambda_0]\wt \right)_{L^2} \\
 && - \Re \left((\varrho^\mu \tilde N_b [\varrho^\mu, \Lambda_0] - [\Lambda_0, \varrho^\mu]\tilde N_b \varrho^\mu) \ut, \wt \right)_{L^2} \Big\} \\
 I_2 &=& -\frac b2 \Big\{ ([\tilde L_{+, b}, \Lambda_0]\varrho^\mu \ut, \varrho^\mu\ut)_{L^2}+   ([\tilde L_{-, b}, \Lambda_0]\varrho^\mu\wt, \varrho^\mu\wt)_{L^2} \\
 &&-2 \Re ([\tilde N_b, \Lambda_0]\varrho^\mu\ut, \varrho^\mu\wt)_{L^2} \Big\} \\
  I_3 &=&\Re \left( (L_{-, b} \varrho^\mu \tilde L_{+, b} \varrho^\mu - \varrho^\mu \tilde L_{-, b} \varrho^\mu L_{+, b} ) \ut, \wt  \right)_{L^2} \\
 &+& \Re \left( (-N_b \varrho^\mu \tilde L_{+,b} \varrho^\mu + \varrho^\mu \tilde N_b \varrho^\mu L_{+, b}) \ut, \ut  \right)_{L^2} \\
 &+& \Re \left( (N_b \varrho^\mu \tilde L_{-,b} \varrho^\mu - \varrho^\mu \tilde N_b \varrho^\mu L_{-, b}) \wt, \wt  \right)_{L^2} \\
 I_4 &=& \Re (\varrho^\mu \ut,  \tilde L_{+, b} \varrho^\mu f - \tilde N_b \varrho^\mu g)_{L^2} + \Re (\varrho^\mu \wt, \tilde L_{-, b} \varrho^\mu g - \tilde N_b \varrho^\mu f)_{L^2} 
\eee
Now we estimate each term. 

\mbox{}

\textit{$I_1$ term.} Note that 
$$[\Lambda_0, \varrho^\mu] = -\mu \varrho^\mu \left(1 - (1+r^4)^{-1}\right),$$
we have
\bee
 I_1 &=& -b\mu E_{b, \mu} + b\mu \Big\{ \Re\left(\tilde L_{+, b} \varrho^\mu \ut, (1+r^4)^{-1} \varrho^\mu\ut \right)_{L^2} \\ 
 &&+ \Re\left(\tilde L_{-, b} \varrho^\mu \wt, (1+r^4)^{-1} \varrho^\mu \wt \right)_{L^2} 
  - 2\Re \left(\tilde N_b \varrho^\mu   \ut, (1+r^4)^{-1} \varrho^\mu  \wt \right)_{L^2} \Big\}
\eee
Thus one easily obtains
\be
 \left| I_1 + b\mu E_{b, \mu}\right| \lesssim b\mu \left( \|\varrho^\mu \ut\|_{L^2_{\la x \ra^{-2}}}^2 + \|\varrho^\mu \ut\|_{\dot H^1}^2 + \|\varrho^\mu \wt\|_{L^2_{\la x \ra^{-2}}}^2 + \|\varrho^\mu \wt\|_{\dot H^1}^2  \right) \label{eqestI1}
\ee

\mbox{}

\textit{$I_2$ term.} Decompose
\bee
 I_2 &=& -\frac b2 \Big\{ (L_1 \varrho^\mu \ut, \varrho^\mu\ut)_{L^2}+   (L_2 \varrho^\mu\wt, \varrho^\mu\wt)_{L^2} \Big\} \\
 &-& \frac b2 \Big\{ ([ W_+ -\psi_b W_{+,b}, \Lambda_0]  \varrho^\mu \ut, \varrho^\mu\ut)_{L^2}+ ([W_- -\psi_b W_{-,b}, \Lambda_0] \varrho^\mu\wt, \varrho^\mu\wt)_{L^2} \\
 &&- 2 \Re ([\tilde N_b, \Lambda_0]\varrho^\mu\ut, \varrho^\mu\wt)_{L^2} \Big\} \\
 &-& \frac b2 \Big\{ (\tilde \psi_b \varrho^\mu \ut, \varrho^\mu\ut)_{L^2}+   (\tilde \psi_b  \varrho^\mu\wt, \varrho^\mu\wt)_{L^2} \Big\}  =: \sum_{k=1}^3 I_{2, k}
\eee

For $I_{2, 1}$, noting that \eqref{eqsolitondecay}-\eqref{eqdecayeigen} imply $\| (\rho^\mu - 1) F\|_{L^2_{\la x \ra^2}} \lesssim \mu$ for $F \in N_g(\calH_0)$, so
\eqref{eqaout}-\eqref{eqaowt} yield that
\bee
\left|(\varrho^\mu \ut, Q)_{L^2}\right| + \left|(\varrho^\mu \ut, Q_1)_{L^2}\right| + \left|(\varrho^\mu \ut, xQ)_{L^2}\right| \lesssim (|\l|^{-4} b^\frac 16 + |\mu| ) \| Y_0 \|_{L^2_{\la x \ra^{-2+|\mu|}}}, \\
\left|(\varrho^\mu \wt, Q_1)_{L^2}\right| + \left|(\varrho^\mu \wt, Q_2)_{L^2}\right| + \left|(\varrho^\mu \wt, \nabla Q)_{L^2}\right| \lesssim (|\l|^{-4} b^\frac 16 + |\mu| ) \| Y_0 \|_{L^2_{\la x \ra^{-2+|\mu|}}}.
\eee
When $|\mu| \ll 1$, this almost orthogonal condition and Proposition \ref{propspecprop} implies coercivity of $-I_{2,1}$. Also note that from \eqref{eqsolitondecay}, \eqref{eqWjbdiffest} and Proposition \ref{propQbasymp},
\bee
\sum_{\pm} \left| x\cdot\nabla(W_\pm - \psi_b W_{\pm, b}) \right| + \left|x\cdot\nabla \tilde N_b\right| \lesssim b^\frac 16 \la r \ra^{-4},
\eee 
so $I_{2, 2}$ can be absorbed into $I_{2, 1}$ to obtain
\bee
 I_{2,1} + I_{2,2} \lesssim -b\left( \|\varrho^\mu \ut\|_{L^2_{\la x \ra^{-\mu_d}}}^2 + \|\varrho^\mu \ut\|_{\dot H^1}^2 + \|\varrho^\mu \wt\|_{L^2_{\la x \ra^{-\mu_d}}}^2 + \|\varrho^\mu \wt\|_{\dot H^1}^2  \right).
\eee
Thanks to $\tilde \psi_b \ge 0$, we can combine the negativity of $I_{2, 3}$ to write
\be
I_2 \lesssim -b \left(  \|\varrho^\mu \ut\|_{L^2_{\la x \ra^{-\mu_d}}\cap \dot H^1 }^2  + \| \tilde \psi_b^\frac 12 \varrho^\mu \ut \|_{L^2}^2 + \|\varrho^\mu \wt\|_{L^2_{\la x \ra^{-\mu_d}}\cap \dot H^1 }^2  + \| \tilde \psi_b^\frac 12 \varrho^\mu \wt \|_{L^2}^2   \right)
\label{eqestI2}
\ee

\mbox{}

\textit{$I_3$ term.} We claim 
\be
|I_3| \lesssim b^\frac 75 \left(  \|\varrho^\mu \ut\|_{L^2_{\la x \ra^{-\mu_d}}\cap \dot H^1 }^2  + \| \tilde \psi_b^\frac 12 \varrho^\mu \ut \|_{L^2}^2 + \|\varrho^\mu \wt\|_{L^2_{\la x \ra^{-\mu_d}}\cap \dot H^1 }^2  + \| \tilde \psi_b^\frac 12 \varrho^\mu \wt \|_{L^2}^2   \right).
\label{eqestI3}
\ee
To begin with, compute
\bee
  && L_{-, b} \varrho^\mu \tilde L_{+, b} \varrho^\mu - \varrho^\mu \tilde L_{-, b} \varrho^\mu L_{+, b} \\
  &=& [-\Delta, \varrho^\mu] \tilde L_{+, b}\varrho^\mu - \varrho^\mu \tilde L_{-,b} [\varrho^\mu , -\Delta] + \varrho^\mu (L_{-, b} \tilde L_{+, b} - \tilde L_{-, b} L_{+, b}) \varrho^\mu \\
  &=&[-\Delta, \varrho^\mu] \varrho^\mu \tilde L_{+, b} - \tilde L_{-,b} \varrho^\mu [\varrho^\mu, -\Delta] \\
  &+& \varrho^\mu  \left\{ [-\Delta, \psi_b] + (-\Delta) \circ ((1-\psi_b) W_{+, b}) - ((1-\psi_b) W_{-, b}) (-\Delta) \right\}\varrho^\mu
\eee
Thus 
\begin{align}
 &\left|  \left( (L_{-, b} \varrho^\mu \tilde L_{+, b} \varrho^\mu - \varrho^\mu \tilde L_{-, b} \varrho^\mu L_{+, b} ) \ut, \wt  \right)_{L^2}  \right|  \nonumber \\
 \lesssim& \| \varrho^\mu [-\Delta, \varrho^\mu] \wt \|_{L^2}  \left( \| \ut \|_{\dot H^2} + \| \psi_b \ut\|_{L^2}  \right) +  \| \varrho^\mu [-\Delta, \varrho^\mu] \ut \|_{L^2}  \left( \| \wt \|_{\dot H^2} + \| \psi_b \wt\|_{L^2}  \right) \label{eqL-bL+bcomterm1} \\
 +& \| |\pa_r^2 \psi_b|^\frac 12 \varrho^\mu \ut\|_{L^2} \| |\pa_r^2 \psi_b|^\frac 12 \varrho^\mu \wt\|_{L^2} + \| |r^{-1}\pa_r \psi_b|^\frac 12 \varrho^\mu \ut\|_{L^2}\| |r^{-1}\pa_r \psi_b|^\frac 12 \varrho^\mu \wt\|_{L^2} \label{eqL-bL+bcomterm2} \\
 +& \| (\pa_r \psi_b) \varrho^\mu \wt \|_{L^2} \| \nabla (\varrho^\mu \ut) \|_{L^2} \label{eqL-bL+bcomterm3}\\
 +&  \| \nabla (\varrho^\mu \wt) \|_{L^2} \| \nabla ((1-\psi_b)W_{+,b} \varrho^\mu \ut) \|_{L^2} +  \| \nabla (\varrho^\mu \ut) \|_{L^2} \| \nabla ((1-\psi_b)W_{-,b} \varrho^\mu \wt) \|_{L^2} \label{eqL-bL+bcomterm4} 
\end{align}

Now we check \eqref{eqL-bL+bcomterm1}-\eqref{eqL-bL+bcomterm4} can be bounded by RHS of \eqref{eqestI3}. For \eqref{eqL-bL+bcomterm1}, we notice that $ \| \varrho^\mu [-\Delta, \varrho^\mu] \wt \|_{L^2} \lesssim  |\mu| \| \varrho^\mu \wt\|_{\dot H^1 \cap L^2_{\la x \ra^{-\mu_d}}}$, $\| \psi_b \ut \|_{L^2} \lesssim R^{\mu_d - \mu} \| \varrho^{\mu} \ut \|_{L^2_{\la x \ra^{-\mu_d}}}$ and that from Lemma \ref{lemapriorieigen} (1), \eqref{eqdefuwt2}, \eqref{eqgammadeltabdd} and \eqref{equwutwtequiv} that
\[ 
\begin{split}  \| \wt \|_{\dot H^2} &\lesssim b^{-1-d} \| \wt \|_{\dot H^1_{\la x \ra^{-1}} \cap L^2_{\la x \ra^{-2}}} + \left( \| \wt \|_{L^2_{\la x \ra^{-2}}} +  \| \ut \|_{L^2_{\la x \ra^{-2}}} \right) (\| (\tilde Q_b)_1 \|_{\dot H^2} + \| \tilde Q_b\|_{\dot H^2} )    \\
&\lesssim b^{-1-d} \left(   \|\varrho^\mu \ut\|_{L^2_{\la x \ra^{-\mu_d}}\cap \dot H^1 } +  \|\varrho^\mu \wt\|_{L^2_{\la x \ra^{-\mu_d}}\cap \dot H^1 } \right)
\end{split}
\]
where $\| (\tilde Q_b)_1 \|_{\dot H^2} + \| \tilde Q_b\|_{\dot H^2}  \lesssim 1$ follows from Proposition \ref{propQbasymp}. Hence \eqref{eqL-bL+bcomterm1} is controlled by RHS of \eqref{eqestI3} thanks to $|\mu| (b^{-1-d} + R^{\mu_d -\mu}) \lesssim b^\frac 75$. For \eqref{eqL-bL+bcomterm2}-\eqref{eqL-bL+bcomterm3}, we apply \eqref{eqpsibrr}, $|r^{-1} \pa_r \psi_b|^\frac 12 + |\pa_r \psi_b| \lesssim  R^{-1} \tilde \psi_b^\frac 12$ and $R^{-1} = b^2$. Lastly, for \eqref{eqL-bL+bcomterm4}, we exploit the smallness of $|W_{\pm, b}| + |\pa_r W_{\pm, b}| \lesssim b^2 r^{-2}$ for $r \ge b^{-2}$ from Proposition \ref{propQbasymp}. 

%
%
%

Similarly, we compute
\bee
   && -N_b \varrho^\mu \tilde L_{\pm,b} \varrho^\mu + \varrho^\mu \tilde N_b \varrho_\mu L_{\pm, b} \\
  &=& \varrho^\mu \tilde N_b [ \varrho^\mu, L_{\pm, b}] + \varrho^\mu (-N_b \tilde L_{\pm,b} - \tilde N_b L_{\pm, b}) \varrho^\mu \\
  &=& \varrho^\mu \tilde N_b (\Delta \varrho^\mu + 2\nabla \varrho^\mu \cdot \nabla) - \varrho^\mu (1-\psi_b) N_b (-\Delta) \circ \varrho^\mu
\eee
and estimate
\bee
  \left| \left( (-N_b \varrho^\mu \tilde L_{+,b} \varrho^\mu + \varrho^\mu \tilde N_b \varrho^\mu L_{+, b}) \ut, \ut  \right)_{L^2}  \right| + \left| \left( (N_b \varrho^\mu \tilde L_{-,b} \varrho^\mu - \varrho^\mu \tilde N_b \varrho^\mu L_{-, b}) \wt, \wt  \right)_{L^2}  \right| \\
  \lesssim \left(|\mu| + b^2 \right)\left( \| \varrho^\mu \ut\|_{L^2_{\la x \ra^{-2}} \cap \dot H^1}^2 +  \| \varrho^\mu \wt\|_{L^2_{\la x \ra^{-2}} \cap \dot H^1}^2 \right).
\eee
So we have verified \eqref{eqestI3}. 

\mbox{}

\textit{$I_4$ term.} Since the support of $f, g$ are within $B_{8b^{-1}} \subset B_R$, so we have 
\bee
I_4 &=&  \Re (\varrho^\mu \ut,   L_{+, b} \varrho^\mu f -  N_b \varrho^\mu g)_{L^2} + \Re (\varrho^\mu \wt,  L_{-, b} \varrho^\mu g - N_b \varrho^\mu f)_{L^2} \\
&=& \Re (\varrho^\mu \ut,  [-\Delta,\varrho^\mu] f)_{L^2} + \Re (\varrho^\mu \wt, [-\Delta,\varrho^\mu] g )_{L^2} \\
&+&  \Re (\varrho^{2\mu} \ut,   L_{+, b}  f -  N_b g)_{L^2} + \Re (\varrho^{2\mu} \wt,  L_{-, b}  g - N_b f)_{L^2}
\eee
Hence using \eqref{eqresfgest} and \eqref{eqbddfg}, we have
\be|I_4| \lesssim (|\mu| + b^2) \left( \| \varrho^\mu \ut \|_{L^2_{\la x \ra^{-2+|\mu|}}}^2 +  \| \varrho^\mu \wt \|_{L^2_{\la x \ra^{-2+|\mu|}}}^2  \right).   
\label{eqestI4}
\ee

\mbox{}

Finally, combine \eqref{eqestI1}, \eqref{eqestI2}, \eqref{eqestI3} and \eqref{eqestI4}, we obtain
\bee
&&  \frac 12 e^{-2b(\mu - s_c) t} \pa_t  ( e^{2b(\mu - s_c) t} E_{b, \mu}) \\
  &\lesssim & -b \left(  \|\varrho^\mu \ut\|_{L^2_{\la x \ra^{-\mu_d}}\cap \dot H^1 }^2  + \| \tilde \psi_b^\frac 12 \varrho^\mu \ut \|_{L^2}^2 + \|\varrho^\mu \wt\|_{L^2_{\la x \ra^{-\mu_d}}\cap \dot H^1 }^2  + \| \tilde \psi_b^\frac 12 \varrho^\mu \wt \|_{L^2}^2   \right)
\eee
But with \eqref{eqexplicitXt2}, we know $E_{b, \mu}(t) = e^{-2\Im \l t}E_{b, \mu}(0)$, so that the left hand side is $0$. That implies $\ut = \wt = 0$ and by \eqref{equwutwtequiv} we conclude the rigidity $Y_0 = 0$. 

\end{proof}




\appendix

\section{Proof of scalar spectral properties}\label{appA1}

\begin{proof}[Proof of Proposition \ref{propspecprop}]
For (1), the coercivity under \eqref{eqortho1} is exactly \cite[Lemma 2.2]{chang2008spectra} (also see \cite{MR0783974}). With orthogonal conditions \eqref{eqortho2}, we proceed as \cite[Lemma 2.3]{MR2150386}. Let $u, w$ satisfy \eqref{eqortho2} and define the auxiliary functions
\[ \ut = u - \a Q_1,\quad \wt = w - \beta Q  \]
with 
\[ \a = \frac{(u, |x|^2 Q)_{L^2}}{(Q_1, |x|^2 Q)_{L^2}} = -  \frac{(u, |x|^2 Q)_{L^2}}{\| xQ\|_{L^2}^2},\quad \beta = \frac{(u, \rho)_{L^2}}{(Q, \rho)_{L^2}} = \frac{2(u, \rho)_{L^2}}{\| xQ \|_{L^2}^2}. \]
Then $u, w$ satisfy \eqref{eqortho1} since $Q_1 \perp_{L^2} Q, xQ$, and then 
\be
(L_+ u, u)_{L^2} = (L_+ \ut, \ut) \gtrsim \|\ut \|_{H^1}^2,\quad (L_- w, w)_{L^2}= (L_- \wt, \wt) \gtrsim \|\wt \|_{H^1}^2. \label{equwuwuw}
\ee
where the first equivalences follow $L_+ Q_1 = -2Q$, $L_- Q = 0$ and $u \perp_{L^2} Q$.

Moreover, from $(u, Q_1)_{L^2} = 0$ and $(w, Q_2)_{L^2} = 0$, we also have
\[ \a = -\frac{(\ut, Q_1)_{L^2}}{\| Q_1\|_{L^2}^2},\quad \beta = - \frac{(\wt, Q_2)_{L^2}}{(Q, Q_2)_{L^2}} = \frac{(\wt, Q_2)_{L^2}}{\| Q_1 \|_{L^2}^2} \]
so $\| u \|_{H^1} \sim \| \ut \|_{H^1}$. This combined with \eqref{equwuwuw} yields \eqref{eqcoerLpm}. 

\mbox{}

For (2), the conclusion in \cite[Theorem 1.1]{MR2150386} (also see earlier works \cite{MR2252148,MR3841347}) can be written as for $u, w$ satisfying \eqref{eqortho2}, we have
\[ (L_1u, u) \gtrsim \left| \begin{array}{ll}
   \| u \|_{\dot H^1(\RR^d)}^2 + \| e^{-\frac{19}{20}|x|} u   \|_{L^2(\RR^d)}^2  &  d = 1 \\
    \| u \|_{\dot H^1(\RR^d)}^2 + \| e^{-\frac 12|x|} u   \|_{L^2(\RR^d)} &  2 \le d \le 10
\end{array}\right. \]
and similarly for $(L_2 w, w)$. The stronger norm in \eqref{eqcoerL12} comes from Hardy inequality. The case $d \ge 3$ follows from the standard one. The $d=1$ case uses (see for example \cite{kufner1990hardy})
\bea
  \int_a^b |u'(x)|^2 dx \gtrsim \int_a^b \frac{|u(x)|^2}{(\min\{x-a, b-x\})^2} dx,\quad u \in C^\infty_c([a, b])
\eea
where the constant is independent of $a, b$, 
which yields
\be \| u' \|_{L^2(\RR)}^2 + \| u \|_{L^2([-1, 1])}^2 \gtrsim \| \la x \ra^{-1} u\|_{L^2(\RR)}^2 \label{eq1DHardy}\ee
 by taking $a = 0$ and let $b \to \pm\infty$.
For $d=2$, we claim the non-sharp Hardy inequality 
\be \| r^{-\a} u \|_{L^2(\RR^2 - B_R^{\RR^2})} \lesssim_{R, \a} \| u \|_{\dot H^1(\RR^2 - B_R^{\RR^2})},\quad \forall R>0, \a > 1 \label{eq2DHardy}
\ee
which implies $\mu_2 = \frac{11}{10}$. Indeed,
for $u \in C^\infty_c (\RR^2 - B_R^{\RR^2})$, consider $S_\a u := \frac{x}{|x|^{1+\a}} u$. So for any $k \in \RR$, 
\bee
0 \le \int_{|x| \ge R} |S_\a u + k \nabla u |^2 dx = \int_{|x| \ge R} |S_\a u|^2 + 2k \int_{|x| \ge R} \nabla u \cdot \S_\a u + k^2 \int_{|x| \ge R} |\nabla u|^2.
\eee
Compute
\bee
\int_{|x| \ge R} \nabla u \cdot \S_\a u  &=& \int_{|x| \ge R} \frac{x\cdot \nabla |u|^2}{2|x|^{1+\a}} = \frac{\a-1}2 \int_{|x| \ge R} \frac{|u|^2}{|x|^{1+\a}} \\
&\ge&  \frac{\a-1}2 \cdot R^{-(\a - 1)}\int_{|x| \ge R} \frac{|u|^2}{|x|^{2\a}}.
\eee
Plug into the first inequality, we obtain
\[ k^2 \int_{|x| \ge R} |\nabla u |^2 \ge\left( -k (\a-1)R^{-(\a - 1)} - 1\right) \int_{|x| \ge R} \frac{|u|^2}{|x|^{2\a}}. \]
So with $\a > 1$ and $R > 0$, \eqref{eq2DHardy} follows by taking $k \ll -1$.

\mbox{}

Next, the improvement of $\mu_2$ with vanishing average condition on almost every sphere comes from spherical harmonic decomposition, which is Fourier transform in 2D (see for example \cite[Chap. IV, Sec. 1]{MR0304972})
\[ \dot H^1(\RR^2) = \bigoplus_{l \ge 0} \dot H^1_l, \quad u(re^{i\theta}) = \sum_{l \ge 0} u_l(r) e^{ il \theta}.\]
Notice that the vanishing average ensures $u_1 = 0$, resulting in
\[ \| u \|_{\dot H^1}^2 = \sum_{l \ge 1} \| u_l(r) e^{ i l\theta} \|_{\dot H^1}^2 = \sum_{l \ge 1}\left( \| u_l \|_{\dot H^1}^2 + l^2 \| r^{-1} u_l \|_{L^2}^2 \right) \ge \| r^{-1} u \|_{L^2}^2.   \]

\mbox{}

Finally, we consider the relaxation of orthogonality condition \eqref{eqortho2}. We only prove \eqref{eqcoerL12} for $u$ since the treatments for $w$ and for \eqref{eqcoerLpm} are almost the same. 

Consider $\ut = u - \a Q - \beta Q_1 - \sum_i \gamma_i x_i Q$ such that $\ut \perp_{L^2} Q, Q_1, xQ$. Then since $Q, Q_1, xQ$ are mutually orthogonal and belong to $L^2_{\la x \ra^{\mu_d}}$, we have $|\a|+ |\beta| + \sum_i |\gamma_i| \lesssim \delta_d \| u \|_{L^2_{\la x \ra^{-\mu_d}}}$. Also from $L_1 Q, L_1 Q_1, L_1 xQ \in L^2_{\la x \ra^{\mu_d}}$, we have $|(L_1 u, u)_{L^2} - (L_1 \ut, \ut)_{L^2}| \lesssim \delta_d  \| u \|_{L^2_{\la x \ra^{-\mu_d}}}$. So the coercivity of \eqref{eqcoerL12} still holds when $\delta_d \ll 1$ depending on the original constant of \eqref{eqcoerL12}.  
\end{proof}

\section{A priori estimates for eigenvalues and eigenfunctions of $\calH_b$} 
\label{appB1}

We first introduce the notations for components of $\calH_b$, 
\be
 \calH_{b, free} = \left( \begin{array}{cc} \Delta_b -1 &  \\ &  -\Delta_{-b} +1\end{array} \right),\quad \calV = \left( \begin{array}{cc} W_{1, b} & W_{2, b}  \\ -\overline{W_{2, b}} & -W_{1, b}   \end{array} \right),
\ee
the resolvent of $\Delta_b$
\be R_b^\pm(z) = \pm i\int_0^\infty e^{\pm it \Delta_b} e^{\pm itz} dt,\quad {\rm for}\,\, \mp \Im z < |b|, \label{eqdefRbpmz} \ee
  as a bounded operator from $L^{p'} \to L^{p}$ for some $p \ge 2$ \cite[Lemma 2.4]{li2023stability}, and thereafter the matrix resolvent of $\calH_{b, free}$
  \be \calS_b^-(\l):=- i\int_0^{\infty}e^{ it\calH_{b, free}} e^{- it\l}dt = \left( \begin{array}{ll} -R_b^+(-1-\l) &  \\ & R_{-b}^-(-1+\l)  \end{array} \right). \ee

To begin with, we show the a priori estimate of eigenvalues.

\begin{lemma}[A priori bound of eigenvalue] \label{lemapriorieigenvalue}
For $d \ge 1$, when $s_c \ll 1$ and $\sigma \in (s_c, s_c + b)$, we have
\be 
\sigma_{dist}\left(\calH_b\big|_{(\dot H^\sigma)^2}\right) \cap \left\{ z \in \CC: \Im z < b(\sigma -s_c) \right\}\subset \left\{ z \in \CC: \Im z \ge -M, |\Re z| \le M b^{-d-1}  \right\}.\label{eqapriorieigenvalue}
\ee
for some $M = M(d) \ge 1$ independent of $s_c$. 
\end{lemma}

\begin{proof}[Proof of Lemma \ref{lemapriorieigenvalue}]
This follows a quantitative version of \cite[Lemma 3.2]{li2023stability}, namely the decay of resolvent when $|\l| \gg 1$. As preparation, 
we decompose the matrix potential as $\calV = \calV_1 \calV_2$ with
\be
|\calV_j| \lesssim  |Q_b|^{\frac{p-1}{2}} \lesssim  \left| \begin{array}{ll}
  e^{-\frac{p-1}{4}r} & r \le 4b^{-1}\\
  e^{-\frac{p-1}{b}} (br)^{-1}  & r \ge 4b^{-1}
\end{array}\right. 
\quad {\rm for} \,\,j = 1, 2, \label{eqcalVjbdd}
\ee 
where the constant is independent of $0 < s_c \ll 1$ using Proposition \ref{propQbasymp}. 

Suppose $\calH_b Y_0 = \l Y_0$ with $Y_0 \in \dot H^\sigma$ for $\sigma \le s_c + b^2 \ll 1$. Apply the resolvent $\calS_b^{-}(\l + ibs_c)$ to the eigen equation yields
\be  Y_0= - \calS_b^-(\l + ibs_c) \calV Y_0.  \label{eqeigeneqY01} \ee
Let $X_0 = \calV_2 Y_0 \in L^2$ by $\sigma \le s_c + b^2 \ll 1$ and Sobolev embedding. Then \eqref{eqeigeneqY01} implies
\be (I + \calV_2 \calS_b^- (\l + ibs_c) \calV_1) X_0 = 0, \label{eqeigeneqX01} \ee
and we will show that outside the range of RHS of \eqref{eqapriorieigenvalue} for some $M \gg 1$, we have 
\be \| \calV_2 \calS_b^-(\l + ibs_c) \calV_1 \|_{L^2 \to L^2} \le \frac 12, \label{eqV2SbV1small} \ee
so that the linear operator is invertible in $L^2$ and thus $X_0 = \vec 0$, contradicting to $Y_0 \neq \vec 0$.

Since $\Delta_b$ is self-adjoint in $L^2$, the spectral estimate $ \| \calS_b^-(\l + ibs_c)\|_{L^2 \to L^2} \le (\Im \l + bs_c)^{-1}$ immediately implies \eqref{eqV2SbV1small} when $\Im \l \ll -1$. 

Finally, we consider 
\be  \Im \l \le b^2 ,\quad |\Re \l| \ge M b^{-d-1}. \label{eqRelgg1} \ee
It suffices to estimate the scalar operator $V_2 R_b^+(-E) V_1$ with $E = 1 + \l + ibs_c$ and $V_j$ as some scalar components of $\calV_j$ satisfying \eqref{eqcalVjbdd}. We will partition the integral of \eqref{eqdefRbpmz} into three regions $[0, a]$ and $[a, \infty)$ to estimate, with $a \ll 1$ independent of $b$.

(a) On $[0, a]$. Take $p_1 = \left| \begin{array}{ll} \frac{2d}{d-1} & d \ge 2, \\ \infty & d = 1. \end{array} \right.$
From the decay of $V_1, V_2$, the dispersive estimate \cite[Proposition 2.2]{li2023stability} 
\be \| e^{it\Delta_b}\|_{L^{q'} \to L^{q}} \lesssim \left| \begin{array}{ll} |t|^{-\left(\frac d2 - \frac dq\right)} & |t| \le |b|^{-1} \\ 
		(|b|^{-1} e^{bt})^{- \left( \frac d2 - \frac dq \right)}  & |t| > |b|^{-1}
	\end{array}  \right.\quad {\rm for} \,\, q \in [2, \infty], \label{eqDeltabdispersive} \ee
 and $\Im E \le b^2 + bs_c \le 2b^2$, we compute
\bea
   \left\| V_1 \int_0^a ie^{it\Delta_b} e^{-itE} dt V_2 \right\|_{L^2 \to L^2} \lesssim \int_0^a \| e^{it\Delta_b}\|_{L^{p_1'} \to L^{p_1}} e^{t\Im E}dt \lesssim a^\frac 12 \ll 1
  \label{eqaintest}
\eea

(b) On $[a, \infty)$. Let $p_2 = \frac{4d}{2d-1}$. We first claim 
\bea
 \left\| V_1 \int_a^\infty ie^{it\Delta_b} e^{-itE} dt V_2 \right\|_{L^{p_2} \to L^{p_2'}} &\lesssim_a & b^{-\frac 34} \label{eqV1semiV2est1}\\
 \left\| V_1 \int_a^\infty ie^{it\Delta_b} e^{-itE} dt V_2 \right\|_{L^1 \to L^\infty} &\lesssim_a &  \left(|E|^{-1} + e^{-\frac{p-1}{b}}\right) b^{\frac d2-1} \label{eqV1semiV2est2} 
\eea
Thereafter, Riesz-Thorin interpolation implies 
\be \begin{split}
   \left\| V_1 \int_a^\infty ie^{it\Delta_b} e^{-itE} dt V_2 \right\|_{L^2 \to L^2} &\lesssim_a \left(|E|^{-\frac{1}{2d+1}} + e^{-\frac{p-1}{(2d+1)b}} \right) b^{\frac{\frac d2 - 1}{2d+1}-\frac 34\cdot\frac{2d}{2d+1}} \\
   &\le M^{-\frac{1}{2d+1}} + o_{b \to 0}(1). 
   \end{split}\label{eqaextest}
\ee

Indeed, the first estimate \eqref{eqV1semiV2est1} comes from that $V_1, V_2$ are bounded from $L^{p_2} \to L^{p_2'}$ by their decay \eqref{eqcalVjbdd}, the above dispersive estimate \eqref{eqDeltabdispersive} with $q = p_2$, and that $\Im E \le 2 b^2 \ll \frac b4$. For the second estimate, it suffices to show the $L^\infty_{x, y}$ estimate of its integration kernel
\[ K(x, y) = V_1(x) V_2(y) \int_a^\infty \left( \frac{ e^{bt} - e^{-bt} }{2b} \right)^{-\frac d2} e^{ib\frac{|xe^{-bt} - y|^2}{2(1-e^{-2bt})} } e^{-itE} dt, \]
where we exploited the integral kernel of $e^{it\Delta_b}$ from \cite[(2.3)]{li2023stability}. With the integrand bounded by $C b^\frac d2 e^{(-\frac d2 + \frac{\Im E}{b}) bt}$ in absolute value, 
a direct integration of $t \in [a, \infty)$ yields 
\be |K(x, y)| \lesssim_a V_1(x) V_2(y) b^{\frac d2 - 1}. \label{eqKLinfest1}\ee
Next, we compute  
\[\left|\pa_t \left( \left( \frac{ e^{bt} - e^{-bt} }{2b} \right)^{-\frac d2} e^{ib\frac{|xe^{-bt} - y|^2}{2(1-e^{-2bt})} }\right) \right| \lesssim_a b^{\frac d2} e^{-\frac d2 bt}\left( b + |x|^2 + |y|^2 \right),\quad {\rm for}\,\,t \ge a, \]
and thus an integration by parts with $e^{itE}dt = \frac{d(e^{itE})}{iE}$ leads to 
\be |K(x, y)| \lesssim_a V_1(x) V_2(y) \cdot b^{\frac d2-1} |E|^{-1}(b + |x|^2 + |y|^2) . \label{eqKLinfest2}\ee
Combined with the decay \eqref{eqcalVjbdd}, we can apply \eqref{eqKLinfest2} when $|x| \ge 4b^{-1}$ or $|y| \ge 4b^{-1}$ and \eqref{eqKLinfest1} otherwise to obtain $\| K \|_{L^\infty_{x, y}} \lesssim_a  \left(|E|^{-1} + e^{-\frac{p-1}{b}}\right) b^{\frac d2-1}$, which leads to \eqref{eqV1semiV2est2}.

Combining \eqref{eqaintest} and \eqref{eqaextest} with $a \ll 1$ determined by \eqref{eqaintest} and thereafter $M \gg 1$, we derive $\| V_2 R_b^+(-E) V_1 \|_{L^2 \to L^2} \ll 1$ under the condition \eqref{eqRelgg1}. And that concludes the a priori bound of eigenvalue \eqref{eqapriorieigenvalue}. 
\end{proof}

\mbox{}

Next, we need the following lemma showing that the resolvent $R_b^\pm(z)$ can trade decay into smoothness. 


\begin{lemma}[Smoothing of resolvent]\label{lemsmoothres}
    Let $d \ge 1$, $b > 0$, $\Im z > -b\min \{ 1, \frac d2\}$. For  $\delta \in [0, 1]$ and $\frac{\Im z}{b} +\delta >0$,
     \be
    \| R_{\pm b}^\pm( \pm z) f \|_{\dot H^{ \delta}} \lesssim \left( b^{-1} + ( \Im z + b\delta)^{-1}\right) \la \Re z\ra^{\max \{\frac{2\delta - 1}{4}, 0 \}} \| f \|_{L^2_{\la x \ra}}.\label{eqressmooth1}
    \ee
     where the constant is independent of $b, z, \delta$. 
\end{lemma}

\begin{proof}  For simplicity, we only prove the case for $R^+_b(z)$. 
From the boundedness of $R_b^+$ in \cite[Lemma 2.4]{li2023stability} and density argument, it suffices to estimate for Schwartz functions $f \in \mathscr{S}(\RR^{d})$, and the Fourier representation in \cite[Lemma 2.4]{li2023stability}
\be \label{eqRb+f} (\widehat{R_b^+(z)f})(r\omega) = \frac{i}{b}r^{-\frac d2 + \frac{\Im z}{b}} e^{\frac{ir^2}{2b}-\frac{i\Re z}{b} \ln r} \int_r^\infty \rho^{\frac d2 - 1 - \frac{\Im z}{b}} e^{-\frac{i\rho^2}{2b}+\frac{i\Re z}{b} \ln \rho} \hat{f}(\rho \omega) d\rho.
	    	\ee
is valid. To begin with, the $\delta = 0$ case of \eqref{eqressmooth1} follows from $\sigma(\Delta_b\big|_{L^2}) = \RR$ that 
\[ \| R_b^+(z) f \|_{L^2} \le (\Im z)^{-1} \| f \|_{L^2}. \]

\underline{1. Quadratic integral estimate.} We first claim the following two general estimates for $\mu > 0$, $A > 0$ and $F\in L^2([A, \infty))$, $G \in L^2([0, A])$: 
\bea
 \int_A^\infty r^{-1+2\mu} \left(\int_r^\infty \rho^{-\frac 12 -\mu} F(\rho) d\rho \right)^2dr &\le& \mu^{-2} \int_A^\infty F^2(\rho) d\rho. \label{eqquaintest1} \\
 \int_0^A r^{-1+2\mu} \left(\int_r^A \rho^{-\frac 12 -\mu} G(\rho) d\rho \right)^2dr &\le& \mu^{-2} \int_0^A G^2(\rho) d\rho.  \label{eqquaintest2}
\eea
Their proofs are similar so we only show \eqref{eqquaintest1}, and through rescaling we can further fix $A = 1$. Firstly, by symmetry of the quadratic term, we estimate 
\bee
 {\rm LHS \,\,of\,\,}\eqref{eqquaintest1} &=& 2 \int_1^\infty dr \int_r^\infty d\rho \int_\rho^\infty d\tau r^{-1+2\mu} \rho^{-\frac 12 -\mu} \tau^{-\frac 12 - \mu} F(\rho) F(\tau)\\
 &\le& 2 \int_1^\infty d\tau \int_1^\tau d\rho \left(\int_1^\rho r^{-1+2\mu} dr \right)\rho^{-\frac 12 -\mu} \tau^{-\frac 12 - \mu} |F(\rho)| |F(\tau)|\\
 &\le& \mu^{-1} \int_1^\infty d\tau \int_1^\tau d\rho \rho^{-\frac 12 + \mu} |F(\rho)| \tau^{-\frac 12 - \mu} |F(\tau)|.
\eee
Next, we change the variable as $\tau = e^t$ and $\rho = e^s$, and denote $f(t) := e^{\frac 12 t}|F(e^t)|\mathbbm{1}_{[0, \infty)}(t)$, then a further application of Young's inequality leads to
\[ {\rm LHS \,\,of\,\,}\eqref{eqquaintest1}  \le \mu^{-1}  \left( f , f * (e^{-\mu(\cdot)}\mathbbm{1}_{[0, \infty)}) \right)_{L^2(\RR)} \le \mu^{-2} \| f \|_{L^2(\RR)}^2.\]
That is \eqref{eqquaintest1} noticing that $\| f \|_{L^2(\RR)}^2 = \int_0^\infty |F(e^t)|^2 e^t dt = \int_1^\infty |F(\rho)|^2 d\rho$.

\mbox{}

\underline{2. Proof of \eqref{eqressmooth1} when $|\Re z| \le 16$.} We first derive a rough estimate to verify \eqref{eqressmooth1} with $|\Re z| \le 16$.  From \eqref{eqRb+f}, we plug in the smooth truncation $\chi$ from \eqref{eqdefchiR} and further integrate by parts over $e^{-i\frac{\rho^2}{2b}}$ for the high-frequency integral to obtain
\be
  (\widehat{R_b^+(z)f})(r\omega) = \hat f_{low} (r\omega) +\hat f_{high, 1}(r\omega) +\hat f_{high, 2}(r\omega) \label{eqRb+zfdecomp1}   \ee
  where
  \begin{align} 
  \hat f_{low} &= \frac{i}{b}r^{-\frac d2 + \frac{\Im z}{b}} e^{\frac{ir^2}{2b}-\frac{i\Re z}{b} \ln r} \int_r^2 \rho^{\frac d2 - 1 - \frac{\Im z}{b}} e^{-\frac{i\rho^2}{2b}+\frac{i\Re z}{b} \ln \rho} \chi(\rho) \hat{f}(\rho \omega) d\rho \nonumber
  \\
  \hat f_{high, 1} &= r^{-2} (1-\chi(r)) \hat f(r\omega) \nonumber \\
  \hat f_{high, 2} &= r^{-\frac d2 + \frac{\Im z}{b}} e^{\frac{ir^2}{2b}-\frac{i\Re z}{b} \ln r} \int_r^\infty \rho^{\frac d2 - 2 - \frac{\Im z}{b}} e^{-\frac{i\rho^2}{2b}+\frac{i\Re z}{b} \ln \rho}(1-\chi(\rho)) \nonumber \\
  &\qquad \qquad \cdot \left(\frac{-\chi'}{1-\chi} + (\frac d2 - 2 + i\frac{z}{b}) \rho^{-1} + \pa_\rho \right) \hat{f} (\rho \omega) d\rho.  \nonumber
\end{align}
The estimate of $f_{low}$ follows directly from \eqref{eqquaintest2} with $F(\rho) := \rho^{\frac{d-1}{2} + \delta} |\hat f(\rho \omega)|$ and $\mu = \frac{\Im z}{b} + \delta > 0$ that
\bea
 && \| f_{low}\|_{\dot H^\delta}^2 \lesssim b^{-2} \int_{\SS^{d-1}} \int_0^2 r^{-1 + 2\frac{\Im z}{b} + 2\delta} \left(\int_r^2 \rho^{\frac d2 - 1- \frac{\Im z}{b}} |\hat f(\rho \omega)| d\rho \right)^2 dr d\omega \nonumber \\
 &\le& (\Im z + b\delta)^{-2} \int_{\SS^{d-1}} \int_0^2 |\hat f(r \omega)|^2 r^{d-1+2\delta} dr d\omega \lesssim (\Im z + b\delta)^{-2} \| f \|_{L^2}^2. \label{eqflowest}
\eea
Similarly, we apply \eqref{eqquaintest1} to estimate $f_{high, 2}$ 
\bea
&& \| f_{high, 2}\|_{\dot H^\delta}^2 \nonumber \\ &\lesssim&  (1+ b^{-1} |z|)^2 \int_{\SS^{d-1}}\int_0^\infty r^{-1+2\frac{\Im z}{b} +2\delta} \left( \int_{\max\{1, r  \}}^\infty \rho^{\frac d2 - 2 - \frac{\Im z}b } (|\hat f| + |\pa_\rho \hat f|)(\rho \omega)d\rho \right)^2 drd\omega \nonumber \\
&\lesssim&  (1+ b^{-1} |z|)^2 (b^{-1}\Im z + \delta)^{-2} \int_{|\xi| \ge 1} (|\hat f|^2 + |\nabla_\xi \hat f|^2) |\xi|^{-2+2\delta}d\xi \nonumber \\
&\lesssim&  (b + |z|)^2\cdot  (\Im z + b\delta)^{-2} \| f \|_{L^2_{\la x \ra}}^2.\label{eqfhigh2est}
\eea
Here we exploited $\delta \le 1$. 
Apparently $\| f_{high, 1} \|_{\dot H^\delta} \lesssim \| f \|_{L^2}$. They lead to \eqref{eqressmooth1} when $|\Re z| \le 16$.

\mbox{}

\underline{3. Proof of \eqref{eqressmooth1} when $|\Re z| \ge 16$.}

\textit{Case 1. $\Re z \ge 16$.} Besides $\chi$, we define another smooth truncation near the stationary point $\sqrt{\Re z}$ as $\tilde \chi = \left| \begin{array}{ll}
    1 & |r - \sqrt{\Re z}| \le 1, \\
    0 & |r - \sqrt{\Re z}| \ge 2.
\end{array}\right.$ Extracting the resonance part and integrating by parts over $\rho^{-\frac{\Im z}{b}} e^{-i\frac{\rho^2}{2b}+ i \frac{\Re z}{b}\ln \rho}$ result in
\be
  (\widehat{R_b^+(z)f})(r\omega) = \hat f_{low} (r\omega) + \hat f_{res}(r\omega) +\hat f_{high', 1}(r\omega) +\hat f_{high', 2}(r\omega)  \ee
  where $f_{low}$ is as in the decomposition \eqref{eqRb+zfdecomp1} and the rest terms are
  \begin{align*} 
  \hat f_{res} &= \frac{i}{b}r^{-\frac d2 + \frac{\Im z}{b}} e^{\frac{ir^2}{2b}-\frac{i\Re z}{b} \ln r} \int_r^\infty \rho^{\frac d2 - 1 - \frac{\Im z}{b}} e^{-\frac{i\rho^2}{2b}+\frac{i\Re z}{b} \ln \rho} \tilde \chi(\rho) \hat{f}(\rho \omega) d\rho
  \\
  \hat f_{high', 1} &= \left(r^2 - z \right)^{-1} (1-\chi(r) - \tilde \chi(r)) \hat f(r\omega) \\
  \hat f_{high', 2} &= r^{-\frac d2 + \frac{\Im z}{b}} e^{\frac{ir^2}{2b}-\frac{i\Re z}{b} \ln r} \int_r^\infty \rho^{\frac d2 - \frac{\Im z}{b}} e^{-\frac{i\rho^2}{2b}+\frac{i\Re z}{b} \ln \rho}  \frac{1-\chi- \tilde \chi}{\rho^2 - z } \\
  &\qquad \qquad \cdot \left(\frac{-\chi' - \tilde \chi'}{1-\chi - \tilde \chi} + \frac d2 \rho^{-1} - \frac{2\rho}{\rho^2 - z}
   + \pa_\rho \right) \hat{f} (\rho \omega) d\rho. 
\end{align*}
Here we have used $\supp \chi \cap \supp \tilde \chi = \emptyset$. Then the term $f_{low}$ has been estimated in \eqref{eqflowest}. For $f_{res}$, by Cauchy-Schwartz, 
\bee
 \|  f_{res} \|_{\dot H^\delta}^2 &\lesssim& b^{-2} \int_{\SS^{d-1}} \int_0^{\sqrt{\Re z} + 2} r^{-1+2\frac{\Im z}{b} + 2\delta} \left(\int_{\sqrt{\Re z} - 2}^{\sqrt{\Re z} + 2} \rho^{\frac d2 - 1 - \frac{\Im z}{b} } |\hat f| d\rho  \right)^2 dr d\omega \\
 &\lesssim& b^{-2} (b^{-1} \Im z + \delta)^{-1} (\sqrt{\Re z})^{2(b^{-1} \Im z + \delta)} \cdot \| f \|_{L^2}^2  (\sqrt{\Re z})^{-1-2b^{-1} \Im z }  \\
 &\lesssim& b^{-1} (\Im z + b\delta)^{-1} \| f \|_{L^2}^2 (\Re z)^{-\frac 12+ \delta}
\eee
where we used $\frac{\Im z}{b} + \delta > 0$. Notice that $\sup_{r \ge 0} \left| r^{\delta} (r^2 - z)^{-1} (1-\chi -\tilde \chi)\right| \lesssim 1$ for $\delta \le 1$ we easily have $\| f_{high', 1} \|_{\dot H^\delta} \lesssim \| f \|_{L^2}$. Finally, we estimate 
\bee
 \| f_{high', 2} \|_{\dot H^\delta}^2 &\lesssim& \left\| r^{-\frac d2 + \frac{\Im z}{b} + \delta} \int_{\max \{ 1, r \}}^\infty \rho^{\frac d2 - 2 - \frac{\Im z}{b} } (|\hat f| + |\pa_\rho \hat f|) d\rho \right\|_{L^2(\RR^d)}^2\\
 &+& \left\| r^{-\frac d2 + \frac{\Im z}{b} + \delta} \int_{\frac 12 \sqrt{\Re z}}^{2\sqrt{\Re z}} \rho^{\frac d2 - 1 - \frac{\Im z}{b} } \la\rho - \sqrt z \ra^{-1} (|\hat f| + |\pa_\rho \hat f|) d\rho \right\|_{L^2(B_{2\sqrt{\Re z}})}^2 \\
 &\lesssim& \left(\frac{\Im z}{b} + \delta\right)^{-2} \cdot \left(1 + (\Re z)^{-\frac 12+ \delta}\right) \| f \|_{L^2_{\la x \ra}}^2
\eee
Here the first term can be estimated like $f_{high, 2}$ in \eqref{eqfhigh2est} without the $\frac zb$ coefficient, while the second term is bounded by Cauchy-Schwartz using $\int_{\frac 12 \sqrt{\Re z}}^{2\sqrt{\Re z}} \la \rho - \sqrt z \ra^{-2} d\rho \lesssim 1$. These estimates conclude \eqref{eqressmooth1} with $\Re z \ge 16$.

\textit{Case 2. $\Re z \le - 16$.} In this case, there is no stationary phase. So we only apply the high-low truncation $\chi$ and integrate by parts over $\rho^{\frac d2 - 1 - \frac{\Im z }{b}} e^{-i\frac{\rho^2}{2b}+ i \frac{\Re z}{b}\ln \rho}$. A similar and simpler computation like Case 1 will lead to \eqref{eqressmooth1}. 
\end{proof}

Finally, we are in place to prove Lemma \ref{lemapriorieigen}.

\begin{proof}[Proof of Lemma \ref{lemapriorieigen}]  

(1) Suppose $\calH_b Y_0 = \l Y_0$ with $Y_0 \in \dot H^\sigma$ for $\sigma \le s_c + b^2 \ll 1$. Apply the resolvent $\calS_b^{-}(\l + ibs_c)$ to the eigen equation yields
\be  Y_0= - \calS_b^-(\l + ibs_c) \calV Y_0,  \label{eqeigeneqY01} \ee
and using the commutator relation \cite[Lemma 2.11]{li2023stability}, 
\be \nabla^n Y_0 = - \calS_b^-(\l + ib(s_c-n)) \nabla^n (\calV Y_0). \label{eqeigeneqY02}
\ee
Notice that from Proposition \ref{propQbasymp} and Corollary \ref{corodiffest}, $|\pa_r^n \calV| \lesssim_n \la r \ra^{-2-n}$ for any $n \ge 0$, and that $\pm\Im(-1\pm(\l+ibs_c)) > - b\sigma > -\frac 12 b$. We apply \eqref{eqressmooth1} with $\delta = 1$ to \eqref{eqeigeneqY01} and \eqref{eqeigeneqY02} to see
\bee
 \| Y_0 \|_{\dot H^{n+1}} \lesssim b^{-1} \la \Re \l \ra^{\frac 12} \| \nabla^n (\calV Y_0) \|_{L^2_{\la x \ra}} \lesssim_n b^{-1} \la \Re \l \ra^{\frac 12} \sum_{k = 0}^n \| \nabla^k Y_0 \|_{L^2_{\la x \ra^{-n+k-1}}},\quad \forall\, \,n \ge 0.
\eee
This implies the smoothness $Y_0 \in \cap_{n \ge 0} \dot H^{n+1}$. Combining the $n = 1$ case with $\la \Re \l\ra \lesssim b^{-d-1} \le b^{-2d}$ from  Lemma \ref{lemapriorieigenvalue}, we obtain \eqref{eqeigensmooth}.

\mbox{}

\underline{(2)} As in  \cite[Lemma 2.1]{MR4250747}, for the linear scalar equation
 \[ V'' + \left( \frac{b^2 r^2}{4} - E\right) V = 0 ,\quad {\rm for}\,\,\, |\Re E| \le 2M b^{-d-1},\,\,\, \Im E \in [-2M, 2b^2],  \]
with $M$ from the previous step, we define the approximate solutions $V^\pm_{b, E}$ for $r \ge b^{-d-3}$ (for simplicity, we won't keep the subscript $b, E$ below unless necessary)
\be
\left| \begin{array}{l}
  V^\pm(r) =  b^\frac 12 (2b)^{\mp \frac{\Im E}{b}} e^{\pm i \frac{1 +2\ln(2b)}{2b} + \frac{\theta^\pm(s)}{b}},\quad \theta^\pm = \pm i \theta_0 + b\theta_1^\pm,\quad s = br, \\
  \theta_0(s) = \frac s4 \sqrt{s^2 - 4\Re E} - \ln \left( \sqrt{s^2 - 4\Re E} + s \right) \\
  \theta_1^\pm (s) = - \frac 14 \ln \left(s^2 - 4\Re E\right) \pm \frac{\Im E}{b} \ln \left( \sqrt{s^2 - 4\Re E} + s \right) 
 \end{array} \right. \label{eqconstructVpm}
\ee
Note that here $s \ge b^{-d-2} \gg 4\Re E$. Then $V^\pm: [b^{-d-3}, \infty)$ satisfies
\begin{enumerate}
\item Equation and residuals: 
\be (V^\pm)'' +  \left( \frac{b^2 r^2}{4} - E\right) V^\pm = b^2 f^\pm(br) V^\pm \ee
where 
\be
  f^\pm(s) = [(\theta_1^\pm)']^2 + (\theta_1^\pm)'';\quad \left|\pa_s^k f^\pm(s)\right| \lesssim_k b^{-2} s^{-2-k}\,\,{\rm for}\,\,s \ge b^{-d-2}. \label{eqfpmest}
\ee
\item Asymptotics: for $r \in [b^{-d-3}, \infty)$, 
\be
  |V^\pm| =  r^{-\frac 12 \pm \frac{\Im E}{b}} \left( 1 + O(|\Re E|b^{-2}r^{-2} ) \right) \quad 
  \left|(V^\pm) \mp i \frac{br}{2}V^\pm\right| \lesssim |\Re E| b^{-1} r^{-1}  |V^\pm|.
  \label{eqVpmest}
\ee
\item Wronskian: $\calW(V^+, V^-) = V^+ (V_-)' - (V^+)' V^- = - ib - \frac{2b \Im E}{b^2 r^2 - 4\Re E}$. 
\end{enumerate}
Here the constants for residual terms in \eqref{eqfpmest}, \eqref{eqVpmest} all depend on $M$. 

\mbox{}

Next, we mimic \cite[Proposition 2.1]{MR4250747} to construct admissible fundamental solutions of $(\calH_b - \l)Y_0 = 0$ on $[b^{-d-3}, \infty)$ via the method of variation of constants. Let
\be \Phi_0 := \left( \begin{array}{c} \phi_+ \\ \phi_- \end{array} \right) 
=  r^{\frac{d-1}{2}}
\left( \begin{array}{c}  e^{i\frac{br^2}{4}} Y_0^1 \\ e^{i\frac{br^2}{4}} \overline{Y_0^2} \end{array} \right), \label{eqdefphipm} \ee
then they satisfy
\[\left| \begin{array}{c}
     \left(\pa_r^2 - E_+ + \frac{b^2 r^2}{4}\right) \phi_+ = L(\phi_+, \bar \phi_-)
     \\
     \left(\pa_r^2 - E_- + \frac{b^2 r^2}{4}\right) \phi_- = L(\phi_-, \bar \phi_+)
\end{array}\right. \]
where $E_\pm = 1 \pm (\l + ibs_c)$, and the linear operator $L(f, g) = \left(\frac{(d-1)(d-3)}{4r^2} - W_{1,b}\right) f - e^{i\frac{br^2}{2}} W_{2,b} g$. 
With $V^\pm_{b, E}$ and $f^\pm_{b, E}$ from \eqref{eqconstructVpm} and \eqref{eqfpmest}, we look for $(\phi_+, \phi_-)$ of the form
\be
 \left| \begin{array}{l}
     \pa_r^k \phi_+ = \iota^+_+ \pa_r^k V^+_{b, E_+} + \iota^-_+ \pa_r^k V^-_{b, E_+}, \quad k = 0, 1;\\
     \pa_r^k \phi_- = \iota^+_- \pa_r^k V^+_{b, \bar E_-} + \iota^-_- \pa_r^k V^-_{b, \bar E_-}, \quad k = 0, 1.
 \end{array}\right.
\ee
Plugging this ansatz into the equation yields the evolutionary equation of $\iota^{\pm_1}_{\pm_2}$ 
\be
\left| \begin{array}{c}
  (\iota^\pm_+)' = \mp \frac{V^\mp_{b, E_+}}{\calW(V^+_{b, E_+}, V^-_{b, E_+})}\left(L(\phi_+, \bar \phi_-) + \sum_\pm b^2 \iota^\pm_+ f^\pm_{b, E_+}(br) V^\pm_{b, E_+} \right) \\
  (\iota^\pm_-)' = \mp \frac{V^\mp_{b, \bar E_-}}{\calW(V^+_{b, \bar E_-}, V^-_{b, \bar E_-})}\left(L(\phi_-, \bar \phi_+) + \sum_\pm b^2 \iota^\pm_- f^\pm_{b, \bar E_-}(br) V^\pm_{b, \bar E_-} \right)
\end{array}\right. \label{eqlpmpm1}
\ee
From the asymptotics of $V^\pm$ \eqref{eqVpmest}, definition of $(\phi_+, \phi_-)$ \eqref{eqdefphipm}, and that $Y_0 \in \dot H^1(\RR^d)$ by Lemma \ref{lemapriorieigen} (1), we have $\a^-_\pm \lim_{r\to \infty} \iota^-_\pm(r) = 0$, and $Y_0$ is determined by $\a^+_\pm :=\lim_{r\to\infty} \iota^+_\pm(r) \in \CC$. 

Without loss of generality, we can normalize $\sum_\pm |\a_\pm|^2 = 1$. Integrate \eqref{eqlpmpm1} from $r \to \infty$, 
\be
  \left| \begin{array}{c}
  \iota^\pm_+ = \a^\pm_+ \mp \int_r^\infty \left[ \frac{V^\mp_{b, E_+}}{\calW(V^+_{b, E_+}, V^-_{b, E_+})}\left(L(\phi_+, \bar \phi_-) + \sum_\pm b^2 \iota^\pm_+ f^\pm_{b, E_+}(br') V^\pm_{b, E_+} \right) \right] dr' \\
  \iota^\pm_- = \a^\pm_- \mp \int_r^\infty \left[\frac{V^\mp_{b, \bar E_-}}{\calW(V^+_{b, \bar E_-}, V^-_{b, \bar E_-})}\left(L(\phi_-, \bar \phi_+) + \sum_\pm b^2 \iota^\pm_- f^\pm_{b, \bar E_-}(br') V^\pm_{b, \bar E_-} \right) \right] dr'.
\end{array}\right. \label{eqlpmpm2}
\ee
We solve $\iota^{\pm_1}_{\pm_2}$ by showing RHS is a $\RR$-linear contraction mapping for $\vec \iota := (\iota^+_+, \iota^-_+, \iota^+_-, \iota^-_-): [b^{-d-3}, \infty) \to \CC^4$ with the norm 
\[ \| \vec \iota\|_{\calA_{b, \l}} := \max_\pm \left\{ \| \iota^+_\pm \|_{L^\infty([b^{-d-3}, \infty))}, \| r^{-2\frac{\min \{\Im \l +bs_c, 0\} }{b}} \iota^-_\pm \|_{L^\infty([b^{-d-3}, \infty))} \right\} \]
Indeed, noticing that $\Im E_+ = \Im \bar E_- = \Im \l + bs_c < \frac 14$ and that $|W_{j, b}| \lesssim r^{-2}$ for $j = 1, 2$ on $[b^{-2}, \infty)$ from Proposition \ref{propQbasymp}, we easily see
\bee
 \left| \left(L(\phi_+, \bar \phi_-) + \sum_\pm b^2 \iota^\pm_+ f^\pm_{b, E_+}(br') V^\pm_{b, E_+} \right)\right| \lesssim r^{-2+ \frac{\Im \l}{b}} \| \vec \l \|_{\calA_{b, E}}
\eee
Hence a direct integration implies that RHS of \eqref{eqlpmpm2} is a bounded operator on $\calA_{b, \l}$ with operator norm bounded by $O(b^{-1}b^{(-d-3)\cdot(-\frac 32)}) = O(b^{\frac 32 d+\frac 72}) \ll 1$. 

Consequently, for any $Y_0$, there exists $C_{Y_0} > 0$ such that 
\[  \frac 12 C_{Y_0} r^{-\frac d2 + s_c + \frac{\Im \l}{b}} \le |Y_0 (r)| \le 2 C_{Y_0} r^{-\frac d2 + s_c + \frac{\Im \l}{b}},\quad \forall \, r \ge b^{-d-3}.   \]
Thus, for $\a \in \{ 1, \frac{11}{10}\}$, we have 
\[ \| r^{-\a} Y_0 \|_{L^2(\{ | x|\ge b^{-d-3}\})}^2 \sim C_{Y_0} \int_{b^{-d-3}}^\infty r^{-d+2(s_c + \frac{\Im \l}{b})} r^{d-1-2\a} dr = C_{Y_0} \frac{b^{2(d+3)(\a-s_c - \frac{\Im \l}{b})}}{2(\a-s_c - \frac{\Im \l}{b})},\]
and hence 
\[ \| r^{-1} Y_0 \|_{L^2(\{ | x|\ge b^{-d-3}\})}^2 \sim \| r^{-\frac{11}{10}}Y_0 \|_{L^2(\{ | x|\ge b^{-d-3}\})}^2 b^{-\frac{d+3}{5}}\]
Finally, we can compute 
\bee
 \| r^{-1} Y_0 \|_{L^2(\RR^d - B_{4b^{-1}})}^2 &\lesssim& \| r^{-\frac{11}{10}}b^{-\frac{d+3}{10}} Y_0 \|_{L^2(B_{b^{-d-3}})}^2 + \| r^{-\frac{11}{10}}Y_0 \|_{L^2(\{ | x|\ge b^{-d-3}\})}^2 b^{-\frac{d+3}{5}} \\
 &\lesssim& b^{-\frac{d+3}{5}} \| r^{-\frac{11}{10}} Y_0 \|_{L^2(\RR^d)}^2.
\eee
The $d = 2$ case implies \eqref{eqdelocalrad}.

\end{proof}

\bibliographystyle{plain}
\bibliography{Bib-1}

\end{document}